\documentclass[11pt]{amsart}
\usepackage{fullpage,mathrsfs,amssymb,color,enumerate}
\usepackage{amsmath}  

\numberwithin{equation}{section}

\usepackage{color}
\newcommand{\vanish}[1]{}

\def\bbar#1{\setbox0=\hbox{$#1$}\dimen0=.2\ht0 \kern\dimen0 }

  \newcommand{\FF}{{\mathbb F}}

 \newcommand{\QQ}{{\mathbb Q}}

\newcommand{\ZZ}{{\mathbb Z}}

\def\bbar#1{\setbox0=\hbox{$#1$}\dimen0=.2\ht0 \kern\dimen0 \overline{\kern-\dimen0 #1}}

\DeclareMathOperator{\rad}{rad}

\DeclareMathOperator{\Frob}{Frob}

\DeclareMathOperator{\End}{End}

\DeclareMathOperator{\Hom}{Hom}

\DeclareMathOperator{\Aut}{Aut}
\DeclareMathOperator{\Gal}{Gal}

\newcommand{\GL}{\operatorname{GL}}

\newcommand{\injects}{\hookrightarrow}
\newcommand{\isom}{\simeq}

\newcommand{\isomto}{\overset{\sim}{\rightarrow}}

\def\p{\mathfrak{p}}
\def\q{\mathfrak{q}}
\def\vV{\mathcal{V}}
\def\tT{\mathcal{T}}
\def\conn{{\operatorname{conn}}}
\def\GSp{{\operatorname{GSp}}}

\usepackage{amsthm}

\newtheorem{theorem}{Theorem}[section]
\newtheorem{lemma}[theorem]{Lemma}

\newtheorem{proposition}[theorem]{Proposition}

\theoremstyle{definition}
\newtheorem{definition}[theorem]{Definition}

\theoremstyle{remark}
\newtheorem{remark}[theorem]{Remark}
\newtheorem{remarks}[theorem]{Remarks}

\definecolor{webcolor}{rgb}{0,0,1}
\definecolor{webbrown}{rgb}{.6,0,0}
\usepackage[
        colorlinks,
        linkcolor=webbrown,  filecolor=webbrown,  citecolor=webbrown, 
        backref,
        pdfauthor={Theodore Hui}, 
        pdftitle={A radical characterization of abelian varieties},
]{hyperref}
\usepackage[alphabetic,backrefs,lite]{amsrefs} 
\usepackage{datetime}
\date{\today}

\begin{document}

\title[]{A Radical Characterization of Abelian Varieties}


\author{Theodore Hui}
\address{Department of Mathematics, Cornell University, Ithaca, NY 14853, USA}
\email{hh535@math.cornell.edu}

\begin{abstract}
Let $A$ be a square-free abelian variety defined over a number field $K$. Let $S$ be a density one set of prime ideals $\p$ of $\mathcal{O}_K$.
A famous theorem of Faltings says that the Frobenius polynomials $P_{A,\p}(x)$ for $\p\in S$ determine $A$ up to isogeny. 
We show that the prime factors of $|A(\FF_\p)|=P_{A,\p}(1)$ for $\p\in S$ also determine $A$ up to isogeny over an explicit finite extension of $K$. 
The proof relies on understanding the $\ell$-adic monodromy groups which come from the $\ell$-adic Galois representations of $A$, and the absolute Weyl group action on their weights.
\end{abstract}

\maketitle

\section{Introduction}\label{intro}

Let $A$ be a non-zero abelian variety defined over a number field $K$. Let $\Sigma_K$ be the set of non-zero prime ideals of the ring of integers $\mathcal{O}_K$ of $K$. For each prime $\p\in \Sigma_K$, let $\FF_\p:=\mathcal{O}_K/\p$ be the corresponding residue field. For all but finitely many primes $\p\in \Sigma_K$, $A$ has good reduction modulo $\p$ and such a reduction gives an abelian variety $A_\p$ defined over $\FF_\p$. Let $P_{A,\p}(x)$ be the Frobenius polynomial of $A$ at $\p$ which we will define in \S\ref{galrep}. If $A$ is isogenous to another abelian variety $A'$ also defined over $K$, then one can show that $P_{A,\p}(x)=P_{A',\p}(x)$ for all $\p\in\Sigma_K$ for which $A$ and $A'$ have good reduction.

Let $S$ be a density one subset of $\Sigma_K$ for which $A$ has good reduction.
A theorem of Faltings says that the function $\p\in S\mapsto P_{A,\p}(x)$ determines $A$ up to isogeny \cite[\S5, Corollary 2]{MR861971}, i.e., if $A$ and $A'$ are abelian varieties defined over a number field $K$ such that $P_{A,\p}(x)=P_{A',\p}(x)$ for a density $1$ set of prime ideals $\p\in \Sigma_K$, then $A$ is isogenous to $A'$ (over $K$).
In fact, one can further show that the function $\p\in S\mapsto |A(\FF_\p)|$ determines $A$ up to isogeny; note that this is a weaker condition than in Faltings' theorem since $|A(\FF_\p)|=P_{A,\p}(1)$. This result seems to be unknown and we will give a quick proof in \S\ref{DZ}. 

Let $\Lambda$ be a set of rational primes.
For any integer $n\geq 1$, we define the \textbf{radical of $n$ with respect to $\Lambda$} by
\[
\rad_\Lambda(n):=\prod_{\ell\in\Lambda,\ \ell | n} \ell.
\] 
Note that when $\Lambda$ is the set of all rational primes, $\rad(n):=\rad_\Lambda(n)$ is the usual definition of radical of $n$, i.e., the product of the distinct prime divisors of $n$.\\ 

Now, let $\Lambda$ be a density one subset of rational primes. 
The main goal of this paper is to study if and when 
the function $\p\in S\mapsto \rad_\Lambda|A(\FF_\p)|$ determines $A$ up to isogeny; note that this is an even weaker condition. This problem has already been studied for special classes of $A$, see \S\ref{previous}.\\

The abelian variety $A$ is isogenous to $\prod_i B_i^{e_i}$ with $B_i$ pairwise non-isogenous simple abelian varieties defined over $K$ and $e_i\geq 1$. It is easy to see that 
\[
\rad_\Lambda|A(\FF_\p)|=\rad_\Lambda\left|(\prod_{i}\nolimits B_i)(\FF_\p)\right|
\] 
for all $\p\in \Sigma_K$ for which $A$ has good reduction; it does not depend on the $e_i\geq 1$. So in general, we will not be able to recover the isogeny class of $A$ by studying $\rad_\Lambda|A(\FF_\p)|$ for $\p\in \Sigma_K$.
This motivates the following definition: we say that $A$ is \textbf{square-free} if it is non-zero and $e_i=1$ for all $i$.

Let $\overline{K}$ be a fixed algebraic closure of $K$.
Let $K_A^\conn$ be the minimal extension of $K$ in $\overline{K}$ for which the $\ell$-adic monodromy groups of $A$ are connected, see \S\ref{galrep}. We can also characterize $K_A^\conn$ as the minimal extension of $K$ in $\overline{K}$ for which $K_A^\conn$ is contained in the torsion field $K(A[\ell])$ for all sufficiently large primes $\ell$, see \cite[Theorem 0.1]{MR1441234}.

Our main theorem says that if $A$ is square-free and if we replace $K$ with $K_A^\conn$, then the function $\p\in S\mapsto \rad_\Lambda|A(\FF_\p)|$ determines $A$ up to isogeny. We will give a proof in \S\ref{proofofthm}.

\begin{theorem}\label{mainthm}
Let $A$ be a square-free abelian variety defined over a number field $K$ satisfying $K_A^\conn=K$. Let $\Lambda$ be a density $1$ set of rational primes.
Suppose $A'$ is a square-free abelian variety defined over $K$ for which 
\[
\rad_\Lambda |A(\FF_\p)|=\rad_\Lambda |A'(\FF_\p)|
\]
holds for all $\p\in \Sigma_K$ away from a set of density $0$. Then $A$ is isogenous to $A'$ (over $K$).
\end{theorem}

One can slightly weaken the assumption and study what happens when $\rad_\Lambda|A'(\FF_\p)|$ divides $\rad_\Lambda|A(\FF_\p)|$ for all $\p\in S$. Although it seems stronger, we will deduce the following theorem from Theorem \ref{mainthm} in \S\ref{pfofmaincor}. 

\begin{theorem}\label{maincor}
Let $A$ be a square-free abelian variety defined over a number field $K$ satisfying $K_A^\conn=K$; it is isogenous to $\prod_{i=1}^r B_i$, where $B_i$ are pairwise non-isogenous simple abelian varieties defined over $K$. 
Let $\Lambda$ be a density $1$ set of rational primes. 
Suppose that $A'$ is any abelian variety defined over $K$ for which
\[
\rad_\Lambda|A'(\FF_\p)| \text{ divides } \rad_\Lambda|A(\FF_\p)|
\]
for all $\p\in \Sigma_K$ away from a set of density $0$. Then $A'$ is isogenous to $\prod_{i\in I} B_i^{e_i}$ for some subset $I\subseteq \{1,\ldots,r\}$ and integers $e_i\geq 1$. 
\end{theorem}

\begin{remarks}\label{mainremark}

We do not know whether Theorems \ref{mainthm} and \ref{maincor} hold without the assumption $K_A^\conn=K$. If $A'$ is an abelian variety defined over $K$ such that $\rad_\Lambda |A(\FF_\p)|=\rad_\Lambda |A'(\FF_\p)|$ for a density $1$ set of $\p\in \Sigma_K$, then one can show that we also have $\rad_\Lambda |A(\FF_\mathfrak{P})|=\rad_\Lambda |A'(\FF_\mathfrak{P})|$ for a density $1$ set of $\mathfrak{P}\in \Sigma_{K_A^\conn}$ (see Lemma \ref{densityone}). Theorem \ref{mainthm} then implies that $A$ and $A'$ are isogenous over $K_A^\conn$.

\end{remarks}

The methods we used in proving Theorem \ref{mainthm} also enable us to prove
the following theorem in \S\ref{splitting}.

\begin{theorem}\label{seppower}
Let $A$ be a simple abelian variety defined over $K$ satisfying $K_A^\conn=K$. There is an integer $e\geq 1$ such that $P_{A,\p}(x)$ is equal to the $e$-th power of a separable polynomial for all $\p\in\Sigma_K$ away from a set of density $0$. 
\end{theorem}

\begin{remarks}
We can make the integer $e$ of Theorem \ref{seppower} explicit. Define $D:=\End(A)\otimes_\ZZ \QQ$; it is a division algebra since $A$ is simple. We then have $e=[D:E]^{1/2}$, where $E$ is the center of $D$.
\end{remarks}

\subsection{Some previous results}\label{previous}

First, we recall some earlier known cases which are related to Theorem \ref{mainthm}.
An abelian variety $A$ of dimension $g\geq 1$ defined over a number field $K$ is said to be \textbf{fully of type $\GSp$} if the image $\overline{\rho}_{A,\ell}(\Gal_K)$ of the mod-$\ell$ Galois representation of $A$, which we will define in \S\ref{galrep}, is isomorphic to $\text{GSp}_{2g}(\FF_\ell)$ for sufficiently large primes $\ell$.
Perucca \cite[Theorems 1.1, 1.3]{MR3429320} proved the following theorem which extends earlier results of Hall-Perucca \cite{MR3019751} and Ratazzi \cite[Theorem 1.3]{MR3276321}; we state it in terms of our radical function $\rad_\Lambda$. 

\begin{theorem}\label{perucca}
Let $A$ and $A'$ be abelian varieties defined over a number field $K$. Let $S$ be a set of prime ideals of $\mathcal{O}_K$  of density $1$ for which $A$ and $A'$ have good reduction. 
Let $\Lambda$ be an infinite set of rational primes. Suppose that $\rad_\Lambda |A'(\FF_\p)|$ divides $\rad_\Lambda |A(\FF_\p)|$ for all $\p\in S$.
\begin{enumerate}[(a)]
\item Suppose that each of $A$ and $A'$ is an elliptic curve or an abelian variety fully of type $\GSp$. Then $A$ is isogenous to $A'$. 
\item Suppose that the simple factors of both $A_{\overline{K}}$ and ${A'}_{\overline{K}}$ are an elliptic curve or an abelian variety fully of type $\GSp$. 
Then every simple quotient of $A'_{\overline{K}}$ is also a quotient of $A_{\overline{K}}$.
\end{enumerate}
\end{theorem}

When $\Lambda$ has density $1$, Theorem \ref{perucca} can be deduced from Theorem \ref{maincor}. For example, if $A$ is fully of type GSp (and hence $\End(A)=\End(A_{\overline{K}})=\ZZ$), then one can check that $A$ is squarefree and $K_A^\conn = K$, so Theorem \ref{maincor} applies. 
Note that it is important to assume $\Lambda$ is a density $1$ set of rational primes in Theorem \ref{mainthm} since number fields $F$ will arise in our general proof for which we will need infinitely many $\ell\in \Lambda$ that splits completely in $F$. 

When $A$ and $A'$ are products of fully of type GSp or CM elliptic curves,
the Galois images $\overline{\rho}_{A\times A',\ell}(\Gal_K)$ (see \S\ref{galrep}) can be explicitly computed for all sufficiently large $\ell$; in general, these images are mysterious and we will study them by using the $\ell$-adic monodromy groups $G_{A\times A',\ell}$ in \S\ref{firstpart}. 

We also recall the following result which is related to Theorem \ref{seppower}.
Let $A$ be an absolutely simple abelian variety defined over a number field $K$. Zywina showed that if the Mumford-Tate conjecture holds for $A$, then for a density-one set of primes $\p\in\Sigma_K$, $A_\p$ is isogenous to some power of $B$ where $B$ is an absolutely simple abelian vareity defined over $\FF_\p$ \cite{MR3264675}. Using Honda-Tate theory, one then shows that $P_{A,\p}(x)$ is an $e$-th power of an irreducible polynomial for all $\p\in\Sigma_K$ away from a set of density $0$. 

\subsection{Notation} 
We will always denote by $\ell$ a rational prime. The phrase ``almost all'' refers to elements from a density one subset of the set of interest. For a non-zero polynomial $f(x)\in\QQ[x]$ with factorization $f(x)=c\prod_i p_i(x)^{e_i}$ where $c\in \QQ^\times$ and $p_i(x)$ are monic and irreducible, we define $\rad f(x):=\prod_i p_i(x)$. 

Let $\overline{K}$ be a fixed algebraic closure of $K$.
We denote by $\Gal_K$ the absolute Galois group $\Gal(\overline{K}/K)$ of $K$. For an algebraic group $G$ defined over a field, we will denote by $G^\circ$ the connected component of $G$ which contains the identity element; it is an algebraic subgroup of $G$.  

For a free $R$-module $M$, where $R$ is a ring, we denote by $\GL_M$ the group scheme over $R$ for which $\GL_M(B)=\Aut_B(B\otimes_R M)$ for each $R$-algebra $B$.

\subsection{Overview}

Let $A$ and $A'$ be non-zero abelian varieties defined over a number field $K$ that satisfies $K_{A}^\conn= K$. The idea is to first study the case where $A$ and $A'$ are base extended to $K_{A\times A'}^\conn$ (see Proposition \ref{finaleone}) and then show that $K_{A\times A'}^\conn = K$ (see Proposition \ref{finaletwo}).

In \S\ref{galrep}, we review some basics on the $\ell$-adic representations $\rho_{A,\ell}$ arising from the action of $\Gal_K$ on the $\ell$-power torsion points of an abelian variety $A$ over $K$. To each prime $\ell$, we will associate an algebraic group $G_{A,\ell}$ over $\QQ_\ell$ which is called the $\ell$-adic monodromy group.
The Frobenius polynomials $P_{A,\p}(x)$ arise from the images of $\rho_{A,\ell}$ and we can study them using $G_{A,\ell}$. 
In \S\ref{bgweights}, we will give background on reductive groups and their weights. In \S\ref{bgminuscule}, we will study a result related to Pink's work on minuscule representations.

In general, these monodromy groups $G_{A,\ell}$ are mysterious. However, after assuming that they are connected, i.e., $K_A^\conn = K$, then we know just enough properties about these groups that allow us to prove Theorem \ref{mainthm}.

In \S\ref{firstpart}, after extending $K$ to $K_{A\times A'}^\conn$, we study the mod $\ell$ representations $\overline{\rho}_{A\times A',\ell}$ associated to the abelian variety $A\times A'$ for $\ell\in \Lambda$ and show that $\rad P_{A,\p}(x)=\rad P_{A',\p}(x)$ for all $\p\in S$. In \S\ref{secondpart}, we show that the Frobenius polynomials of non-isogenous simple abelian varieties are relatively prime for almost all $\p\in \Sigma_K$; the proof relies heavily on the results from \S\ref{bgminuscule}. 
In \S\ref{pfoffinaleone}, we will then show how to combine \S\ref{firstpart} and \S\ref{secondpart} to show that $A'$ is isogenous to a product of simple factors of $A$ over $K_{A\times A'}^\conn$. In \S\ref{pfoffinaletwo}, we will further show that $K_{A\times A'}^\conn=K$. This gives the proof of Theorem \ref{mainthm}.

In \S\ref{pfofmaincor}, we will show how to use Theorem \ref{mainthm} to prove Theorem \ref{maincor}.

In \S\ref{splitting}, we will use the tools developed in \S\ref{secondpart} to give the proof of Theorem \ref{seppower}. 

In \S\ref{DZ}, we will give a quick proof of why the function $\p\in S\mapsto |A(\FF_\p)|$ determines $A$ up to isogeny, also as promised in the introduction.

\section{Background}\label{background}

\subsection{Galois representations}\label{galrep}
In this section, we let $A$ be an abelian variety of dimension $g\geq 1$ defined over a number field $K$. For each positive integer $m$, let $A[m]$ be the $m$-torsion subgroup of $A(\overline{K})$; it is a free $\ZZ/m\ZZ$-module of rank $2g$. Fix a prime $\ell$. The \textbf{$\ell$-adic Tate module} of $A$ is the inverse limit  
\[
T_\ell(A):=\varprojlim_e A[\ell^e]
\] 
with respect to the multiplication-by-$\ell$ transition maps $A[\ell^{e+1}]\overset{\times\ell}{\to}A[\ell^e]$; it is a free $\ZZ_\ell$-module of rank $2g$. The absolute Galois group $\Gal_K$ naturally acts on $A[m]$ and hence also on $T_\ell(A)$. We thus have a Galois representation
\[
\rho_{A,\ell}:\Gal_K\to \Aut_{\ZZ_\ell}(T_\ell(A)).
\]

Define $V_\ell(A):=T_\ell(A)\otimes_{\ZZ_\ell} \QQ_\ell$; it is a $\QQ_\ell$-vector space of dimension $2g$. By tensoring up with $\QQ_\ell$ and $\FF_\ell$ respectively, $\rho_{A,\ell}$ induces Galois representations 
\[
\rho_{A,\ell}:\Gal_K\to \Aut_{\QQ_\ell}(V_\ell(A))
\]
and 
\[
\overline{\rho}_{A,\ell}:\Gal_K\to \Aut_{\FF_\ell}(A[\ell]),
\]
respectively. For a prime $\p\in\Sigma_K$ such that $\p\nmid \ell$ and $A$ has good reduction, $\rho_{A,\ell}$ is unramified at $\p$, and the \textbf{Frobenius polynomial} of $\p$ is defined by 
\[
P_{A,\p}(x):=\det(xI-\rho_{A,\ell}(\Frob_\p));
\]
it is a monic polynomial of degree $2g$ with integer coefficients and is independent of $\ell$. Note that $P_{A,\p}(x)$ also agrees with the characteristic polynomial of the Frobenius endomorphism $\pi_{A_\p}$ of $A_\p$, where $A_\p$ is the reduction of $A$ modulo $\p$, i.e., the unique polynomial $P(x)\in \ZZ[x]$ such that the isogeny $n-\pi_{A_\p}$ of $A_\p$ has degree $P(n)$ for all integers $n$. 
We have $P_{A,\p}(1)=\deg(1-\pi_{A_\p})=|A(\FF_\p)|$. 
Note that we also have 
\[
P_{A,\p}(x)\equiv\det(xI-\overline{\rho}_{A,\ell}(\Frob_\p))\pmod{\ell}
\]
for all $\p\nmid \ell$ for which $A$ has good reduction.

Let $\mathcal{G}_{A,\ell}$ be the Zariski closure of $\rho_{A,\ell}(\Gal_K)$ in $\GL_{T_\ell(A)}$; it is a group scheme over $\ZZ_\ell$. The generic fibre $G_{A,\ell}:=(\mathcal{G}_{A,\ell})_{\QQ_\ell}$ agrees with the Zariski closure of $\rho_{A,\ell}(\Gal_K)$ in $\GL_{V_\ell(A)}$; it is an algebraic subgroup of $\GL_{V_\ell(A)}$ called the \textbf{$\ell$-adic algebraic monodromy group} of $A$. The special fibre $H_{A,\ell}:=(\mathcal{G}_{A,\ell})_{\FF_\ell}$ is an algebraic subgroup of $\GL_{A[\ell]}$ and we have $\overline{\rho}_{A,\ell}(\Gal_K)\subseteq H_{A,\ell}(\FF_\ell)$.

Let $K_A^{\text{conn}}$ be the fixed field in $\overline{K}$ of the subgroup $(\rho_{A,\ell})^{-1}(G_{A,\ell}^\circ(\QQ_\ell))$ of $\Gal_K$; it is the minimal Galois extension $L$ of $K$ for which $G_{A_L,\ell}$ is connected (it will equal $G^\circ_{A,\ell}$).

\begin{proposition}\label{KAconn}
The number field $K_A^{\conn}$ is independent of the choice of $\ell$. In particular, $K_A^{\conn}=K$ if and only if all the $\ell$-adic monodromy groups $G_{A,\ell}$ are connected.
\end{proposition}

\begin{proof}
See Serre \cite[\#133 p.17]{MR1730973} and \cite{MR1441234}.
\end{proof}

Note that if $A$ and $A'$ are abelian varieties defined over the number field $K$, then $K_{A\times A'}^{\conn}\supseteq K_A^{\conn}\cdot K_{A'}^{\conn}$.

\begin{proposition}\label{reductive}
Assume that $K_A^{\text{conn}}=K$.
Then the $\ZZ_\ell$-group scheme $\mathcal{G}_{A,\ell}$ is reductive for all sufficiently large $\ell$. The algebraic group $H_{A,\ell}$ is connected and reductive for all sufficiently large $\ell$.
\end{proposition}

\begin{proof}
Reductiveness of $\mathcal{G}_{A,\ell}$ was proved in \cite{MR1441234}; see also \cite[\S1.3]{MR1944805} for a minor correction of the proof.
Connectedness of $H_{A,\ell}$ follows from the reductiveness of $\mathcal{G}_{A,\ell}$ and the connectedness of $G_{A,\ell}$. 
\end{proof}   

For a reductive algebraic group $G$ over a field, we say that $G$ is \textbf{split} if it contains a split maximal torus. One can find a more precise definition in the next section. Here are some properties concerning $H_{A,\ell}$ related to Serre's work which will be useful in \S\ref{firstpart}; in particular, part (a) shows that $\overline{\rho}_{A,\ell}(\Gal_K)$ ``almost equals'' $H_{A,\ell}(\FF_\ell)$. 

\begin{theorem}\label{bgHell}
Assume that $K_A^{\text{conn}}=K$. 
\begin{enumerate}[(a)]

\item There exists a constant $M_A$, depending only on $A$, such that $[H_{A,\ell}(\FF_\ell):\overline{\rho}_{A,\ell}(\Gal_K)]\leq M_A$ for all $\ell$. 

\item For $\ell$ sufficiently large, $H_{A,\ell}$ is connected, reductive and contains the group $\mathbb{G}_m$ of homotheties.

\item There is a finite Galois extension $F$ of $\QQ$ such that $H_{A,\ell}$ is split for all sufficiently large $\ell$ that splits completely in $F$. 
 
\end{enumerate}
\end{theorem}

\begin{proof}
In Serre's 1985--1986 course at the Coll\`{e}ge de France \cite[\#136]{MR1730973}, he constructed for each prime $\ell$ a connected, reductive algebraic subgroup of $\GL_{A[\ell]}=\GL_{T_\ell(A),\FF_\ell}$ that satisfies all the properties as stated in (a) and (b). Wintenberger \cite[\S3.4]{MR1944805} showed that this subgroup is isomorphic to the connected component of $H_{A,\ell}$ when $\ell$ is sufficiently large.
For more details, one can refer to \cite[\#137]{MR1730973}, \cite[\#138]{MR1730973}, \cite{MR1944805} and \cite[Proposition 2.10]{MR3264675}. 

For (c), see \cite[Lemma 3.2]{MR3474973}.
\end{proof}

The following results concerning $\rho_{A,\ell}$ will be useful in \S\ref{secondpart}.

\begin{theorem}[Faltings]\label{bgAGell}
\noindent
\begin{enumerate}[(a)]
\item The Galois representation $\rho_{A,\ell}:\Gal_K\to \Aut_{\QQ_\ell}(V_\ell(A))$ is semisimple.
\item For abelian varieties $A$ and $A'$ defined over $K$, the natural homomorphism 
\[
\Hom(A,A')\otimes_\ZZ \QQ_\ell \to \Hom_{\QQ_\ell[\Gal_K]}(V_\ell(A),V_\ell(A'))
\]
is an isomorphism. In particular, $\End_{\QQ_\ell[\Gal_K]}(V_\ell(A))$ is isomorphic to $\End(A)\otimes_\ZZ \QQ_\ell$. 
\item The group $G^\circ_{A,\ell}$ is reductive.
\end{enumerate}
\end{theorem}

\begin{proof}
See \cite[Theorems 3--4]{MR861971}.
\end{proof}

\begin{lemma}\label{endA}
Assume that $K_A^{\text{conn}}=K$.
\begin{enumerate}[(a)]
\item Then we have $\End(A_{\overline{K}})\otimes_\ZZ \QQ = \End(A)\otimes_\ZZ \QQ$. 
\item For any simple abelian subvariety $B$ of $A$, the abelian variety $B_{\overline{K}}$ is also simple.
\end{enumerate}

\end{lemma}

\begin{proof}
For (a), see \cite[Proposition 2.2 (iii)]{MR3264675}. 
For (b), 
let $B/K$ be a simple abelian subvariety of $A$. 
Suppose $B_{\overline{K}}$ is not simple. Then there exists $\phi\in \End(A_{\overline{K}})$ such that $\phi(A_{\overline{K}})$ is a non-zero proper abelian subvariety of $B_{\overline{K}}$. By (a), $\phi$ is defined over $K$ and so $\phi(A)$ is a non-zero proper abelian subvariety of $B$. This contradicts our assumption that $B$ is simple.
\end{proof}

\subsection{Reductive groups and weights}\label{bgweights}
Let $G$ be a connected reductive group defined over a perfect field $k$ and fix an algebraic closure $\overline{k}$ of $k$. A \textbf{torus} of $G$ is an algebraic subgroup $T\subseteq G$ such that $T_{\overline{k}}$ is isomorphic to $(\mathbb{G}_{m})_{\overline{k}}^r$ for some integer $r\geq 0$. We say that $T$ is \textbf{split} if it is isomorphic to $(\mathbb{G}_{m})_{k}^r$. 
A \textbf{maximal torus} of $G$ is a torus $T$ of $G$ that is not contained in any larger torus of $G$; the torus $T_{\overline{k}}$ is a maximal torus of $G_{\overline{k}}$. 
Since $G$ is a reductive group, any two maximal tori of $G_{\overline{k}}$ are conjugate to each other by some element of $G(\overline{k})$. The \textbf{rank} $r$ of $G$ is the dimension of any maximal torus. We say that $G$ is \textbf{split} if it contains a split maximal torus. 

Fix a maximal torus $T$ of $G$. Denote by $X(T)$ the group of characters $T_{\overline{k}}\to (\mathbb{G}_{m})_{\overline{k}}$; it is a free abelian group of rank $r$. The \textbf{(absolute) Weyl group} of $G$ with respect to $T$ is defined as
\[
W(G,T):=N_G(T)(\overline{k})/T(\overline{k})
\] 
where $N_G(T)$ is the normalizer of $T$ in $G$. For $n\in N_G(T)(\overline{k})$, the homomorphism $\iota_n:T_{\overline{k}}\to T_{\overline{k}}$ defined by $t\mapsto ntn^{-1}$ gives an automorphism $\alpha\mapsto \alpha\circ (\iota_n)^{-1}$ of $X(T)$. This gives a faithful left action of $W(G,T)$ on $X(T)$. 

Suppose we have a representation $\rho:G\to \GL_V$ where $V$ is a finite dimensional vector space over $k$.
For each character $\alpha\in X(T)$, let $V(\alpha)$ be the subspace of $V\otimes_k \overline{k}$ consisting of those vectors $v$ for which $\rho(t)\cdot v = \alpha(t) v$ for all $t\in T(\overline{k})$. We say that $\alpha\in X(T)$ is a \textbf{weight} of $\rho$ if $V(\alpha)\not=0$, and we denote the (finite) set of such weights by $\Omega(\rho)$ or $\Omega(V)$. Note that $W(G,T)$ acts on $\Omega(V)$. We have a decomposition $V\otimes_k \overline{k}=\bigoplus_{\alpha\in\Omega(V)} V(\alpha)$ and hence for each $t\in T(\overline{k})$, the characteristic polynomial of $\rho(t)$ is given by
\[
\det(xI-\rho(t))=\prod_{\alpha\in \Omega(V)}(x-\alpha(t))^{m_\alpha}
\]
where $m_{\alpha}:=\dim_{\overline{k}} V(\alpha)$ is the \textbf{multiplicity} of $\alpha$. Note that $m_\alpha=m_\beta$ if $\alpha$ and $\beta$ are in the same $W(G,T)$-orbit.

\subsection{Weak Mumford-Tate pairs and minuscule representations}\label{bgminuscule}

Let $F$ be a field of characteristic zero. 
Suppose $G$ is a connected reductive algebraic group over $F$ with a faithful representation $\rho: G\injects \GL_U$ where $U$ is a finite dimensional $F$-vector space. We have an isomorphism $X(\mathbb{G}_m)=\ZZ$, where an integer $n\in \ZZ$ corresponds to the character $t\mapsto t^n$.  

\begin{definition}
The pair $(G,\rho)$ is called a \textbf{weak Mumford-Tate pair with weights $\{0,1\}$} if there exists a set of cocharacters $\{\mu:(\mathbb{G}_{m})_{\overline{F}}\to G_{\overline{F}}\}$ such that 
\begin{enumerate}[(i)]
\item $G_{\overline{F}}$ is generated by the images of the $G(\overline{F})$-conjugates of all $\mu$, and
\item the weights of each $\rho\circ \mu$ are in $\{0,1\}$. 
\end{enumerate}
\end{definition}

Fix a maximal torus $T$ of $G$. 
Let $W(G,T)$ be the (absolute) Weyl group of $G$ with respect to $T$. Recall that $W(G,T)$ acts on $\Omega(\rho)\subseteq X(T)$. 
In order to study how $W(G,T)$ acts on $\Omega(\rho)$ when $(G,\rho)$ is a weak Mumford-Tate pair, we will also need the following definition.

\begin{definition}
We say that an irreducible representation $\rho: G\to \GL_U$ is \textbf{minuscule} if the Weyl group $W(G,T)$ acts transitively on the weights of $\rho$, i.e., the weights of $\rho$ form a single orbit under the action of the Weyl group $W(G,T)$.
\end{definition}

See \cite[Ch.VIII \S3]{MR2109105} for an equivalent definition of minuscule using $\Omega(U)$-saturations. 

If $F$ is algebraically closed, then $(G,\rho)$ being minuscule implies that all the weights of $\rho$ must have multiplicity $1$ since there exists a highest weight of multiplicity $1$ for $\rho$ (see for example \cite[\S31.3]{MR0396773}).  
We obtained the proof of the following theorem by collecting ideas from Serre \cite[\S3]{MR563476} and Pink \cite[\S4]{MR1603865}.

\begin{theorem}\label{WMT}
Suppose $G$ is a connected reductive group over $F$ with a faithful representation $\rho: G\injects \GL_U$ where $U$ is a finite dimensional $F$-vector space.
If $(G,\rho)$ is a weak Mumford-Tate pair of weights $\{0,1\}$, then each irreducible representation $V\subseteq U\otimes_F \overline{F}$ of $G_{\overline{F}}$ is minuscule.
\end{theorem}

\begin{proof}
First of all, 
$(G,\rho)$ remains a weak Mumford-Tate pair if we base extend $F$ to $\overline{F}$, so without loss of generality we may assume that $F=\overline{F}$. 

Consider an irreducible subrepresentation $\rho_V:G\to \GL_{V}$ of $\rho$. 
Let $G_0:=Z$ denote the identity component of the center of $G$. 
If $G\not=Z$, let $G_1,\ldots, G_s$ denote the 
minimal closed connected normal subgroups of the derived group $G^\text{der}$ with positive dimension. Each $G_i$ is almost simple. We then have an almost direct product $G = G_0\cdot G_1\cdots G_s$ (see for example \cite[\S27.5]{MR0396773}). 
So multiplication gives a homomorphism
\[
\phi: G_0\times G_1\times \cdots \times G_s \to G = G_0\cdot G_1\cdots G_s
\]
with finite kernel (contained in the center of $G$ since $\text{char}(F)=0$). Moreover, since $\rho_V$ is irreducible, there exists irreducible representations $\rho_i: G_i\to \GL_{V_i}$ for some finite dimensional $F$-vector spaces $V_i$, such that $\rho_V\isom \rho_0\otimes \cdots \otimes \rho_s$. We can assume that $V=V_0\otimes \cdots \otimes V_s$.

For each $i$, choose a maximal torus $T_i\subseteq G_i$. Then $\prod_i T_i$ is a maximal torus of $\prod_i G_i$. Let $T=T_0\cdots T_s$ (i.e., the image of $\prod T_i$ under $\phi$); it is a maximal torus of $G$.  
Let $\Omega(V_i)$ be the set of weights with respect to $\rho_i$.

The homomorphism $\phi$ induces an isomorphism between $W(\prod G_i,\prod T_i)=\prod W(G_i,T_i)$ and $W(G,T)$; this uses that the kernel of $\phi$ is finite and contained inside the center of $\prod G_i$. 
Note that the restriction $\prod T_i\to T$ of $\phi$ induces an isomorphism $X(T)\otimes_\ZZ \QQ \isomto \prod X(T_i)\otimes_\ZZ \QQ$ and gives a bijection $\Omega(V)\isom \prod \Omega_i(V_i)$, for which the actions of $W(G,T)$ and $\prod W(G_i,T_i)$ are compatible. Hence, to show that the representation $\rho_V:G\to \GL_V$ is minuscule, i.e., $W(G,T)$ acts transitively on $\Omega(V)$, it suffices to show that $W(G_i,T_i)$ acts transitively on each $\Omega(V_i)$. 

When $i=0$, $V_0$ is one-dimensional since $G_0$ is a torus. So $W(G_0,T_0)$ acts transitively on  the one element set $\Omega(V_0)$. In particular, when $G=Z$, the theorem is true. 

Assume that $G\not=Z$ and consider $i>0$; note that the kernel of $\rho_i$ is either finite or $G_i$, since $G_i$ is almost simple. If $\ker(\rho_i)$ is $G_i$, then $W(G_i,T_i)$ acts transitively on the one element set $\Omega(V_i)$. 
For each $i$, let $\tilde{G_i}$ be the image of $G_i\injects \prod G_i \overset{\rho_V\circ\phi}{\to} \GL_V$. 
Let $I$ be the set of $i>0$ for which $\ker(\rho_i)$ is finite.  
Fix $i\in I$, we have an isogeny $\phi_i:G_i\to \tilde{G_i}$. The image $\tilde{T}_i$ of $T_i$ under $\phi_i$ is a maximal torus of $\tilde{G}_i$. The isogeny $\phi_i$ induces isomorphisms 
$X(\tilde{T}_i)\otimes_\ZZ \QQ\isomto X(T_i)\otimes_\ZZ \QQ$ and hence
an isomorphism $W(G_i,T_i)\isomto W(\tilde{G}_i,\tilde{T}_i)$.
 
The image of $\rho_V$ is a reductive group with almost direct product decomposition $\tilde{G}_0\cdot \prod_{i\in I}\tilde{G}_i$.
Moreover, the $W(G_i,T_i)$ and  $W(\tilde{G}_i,\tilde{T}_i)$ actions on $\Omega(V_i)$ are compatible with respect to these isomorphisms. Hence, to show that the representation $\rho_V:G\to \GL_V$ is minuscule, it suffices to show that $W(\tilde{G}_i,\tilde{T}_i)$ acts transitively on $\Omega(V_i)$ for each $i\in I$. 

By \cite[\S4]{MR1603865}, since $(G/\ker \rho_V\isom \tilde{G}_0\cdot \prod_{i\in I} \tilde{G}_i,\rho_V)$ is a weak Mumford-Tate pair,
we have $\tilde{G}_0= \mathbb{G}_m$ (i.e., the homotheties) and $(\tilde{G}_0\cdot \tilde{G}_i, \tilde{G}_0\cdot \tilde{G}_i\injects \GL_{V_0\otimes V_i})$ is a weak Mumford-Tate pair for each $i\in I$. 
In \cite[Table 4.2]{MR1603865}, Pink listed all the possibilities for $(\tilde{G}_i,\tilde{G}_i\injects \GL_{V_i})$ and in each case $\tilde{G}_i\injects \GL_{V_i}$ is a minuscule representation. This proves the theorem.
\end{proof}

\begin{remark}
A \textbf{strong Mumford-Tate pair} is a weak Mumford-Tate pair together with the extra condition that all the given cocharacters are contained in a single $\Aut(\overline{F}/F)$-orbit.
In \cite[\S3]{MR563476}, Serre focused on the proof of Theorem \ref{WMT} for strong Mumford-Tate pairs. However, in \cite{MR3351177}, Orr pointed out that this extra condition was not being used in the proof. This is also clear from our discussion above. (Note that Orr considered Mumford-Tate triples instead of Mumford-Tate pairs by making the cocharacter set explicit in his paper.) 
\end{remark}

Let $A$ be any abelian variety defined over a number field $K$.
For every prime $\ell$, let $\iota_{A,\ell}:G_{A,\ell}\injects \GL_{V_\ell(A)}$ be the tautological representation of the $\ell$-adic monodromy group.
On the other hand, the first $\ell$-adic \'{e}tale cohomology group $H:=H^1_{\text{\'{e}t}}(A_{\overline{K}},\QQ_\ell)$ of $A$ is isomorphic to the dual of $V_\ell(A)$. 
The $\Gal_K$-action on $H$ gives a continuous representation 
\[
\rho': \Gal_K \to \GL_H
\]
which is dual to the representation $\rho_{A,\ell}$. 
Let $\iota^{\vee}_{A,\ell}: G_{A,\ell}\injects \GL_H$ be the faithful representation induced by $\rho'$ and the duality. Note that $\rho'(\Gal_K)$ is Zariski dense in $\iota^{\vee}_{A,\ell}(G_{A,\ell})$.
Pink proved the following result in \cite[Theorem (5.10)]{MR1603865}, which will be a main ingredient in the proofs of Theorem \ref{mainthm} and Theorem \ref{seppower}.

\begin{theorem}[Pink]\label{Pink}
Let $A$ be an abelian variety defined over a number field $K$. Suppose $K_{A}^{\conn}=K$.
Then for every prime $\ell$, $(G_{A,\ell},\iota^{\vee}_{A,\ell})$ is a weak Mumford-Tate pair of weights $\{0,1\}$.
\end{theorem}

\begin{proposition}\label{minuscule}
Let $A$ be an abelian variety defined over a number field $K$. Suppose that $K_{A}^{\conn}=K$. Then each irreducible representation $V\subseteq V_\ell(A)\otimes_{\QQ_\ell} \overline{\QQ}_\ell$ of $(G_{A,\ell})_{\overline{\QQ}_\ell}$ is minuscule.
\end{proposition}

\begin{proof}
By Theorem \ref{Pink}, $(G_{A,\ell},\iota^{\vee}_{A,\ell})$ is a weak Mumford-Tate pair of weights $\{0,1\}$ over $\QQ_\ell$.
By Theorem \ref{WMT}, 
each irreducible 
component of the dual representation $\iota^{\vee}_{A,\ell}:(G_{A,\ell})_{\overline{\QQ}_\ell}\injects \GL_{H\otimes_{\QQ_\ell} \overline{\QQ}_\ell}$ is minuscule and therefore the same also holds for $\iota_{A,\ell}$. 
\end{proof}

\section{Radicals of Frobenius polynomials}\label{firstpart}
Let $A$ and $A'$ be non-zero abelian varieties defined over a number field $K$ of dimensions $g$ and $g'$ respectively. 
We assume throughout the section that the $\ell$-adic monodromy groups $G_{A\times A',\ell}$ are connected, i.e., $K_{A\times A'}^\conn = K$. 

Let $S$ be a set of prime ideals of $\mathcal{O}_K$ of density $1$ for which $A$ and $A'$ have good reduction. Suppose there is a density $1$ set $\Lambda$ of rational primes such that $\rad_\Lambda|A(\FF_\p)|$ divides $\rad_\Lambda|A'(\FF_\p)|$ for all $\p\in S$. 
The goal of this section is to build up tools for proving the following result, which will be proved in \S\ref{pfofonepttwo}.

\begin{proposition}\label{onepttwo}
The polynomial $\rad P_{A,\p}(x)$ divides $\rad P_{A',\p}(x)$ for all prime $\p\in S$.
\end{proposition}

\subsection{Setup}
For each prime $\ell$, we define $H_\ell:=(\mathcal{G}_{A\times A',\ell})_{\FF_\ell}$ as in \S\ref{background}. By Proposition \ref{reductive} and the assumption $K_{A\times A'}^\conn=K$, the group
$H_\ell$ is connected when $\ell$ is sufficiently large. 
Recall that we have Galois representations $\overline{\rho}_{A\times A',\ell}: \Gal_K\to H_\ell(\FF_\ell)$.
We can identify $H_\ell$ with a closed algebraic subgroup of $(\mathcal{G}_{A,\ell})_{\FF_\ell}\times (\mathcal{G}_{A',\ell})_{\FF_\ell}$. 

\begin{lemma}\label{threeptthree}
For every $\ell\in \Lambda$ and $(B,B')\in \overline{\rho}_{A\times A',\ell}(\Gal_K)$, if $\det(I-B)=0$, then $\det(I-B')=0$.  
\end{lemma}

\begin{proof}
Take any $(B,B')\in \overline{\rho}_{A\times A',\ell}(\Gal_K)$. By the Chebotarev density theorem, there exists a prime $\p\in S$ with $\p\nmid \ell$ such that $\overline{\rho}_{A\times A',\ell}(\Frob_\p)=(\overline{\rho}_{A,\ell}(\Frob_\p),\overline{\rho}_{A',\ell}(\Frob_\p))$ is conjugate to $(B,B')$ in $H_\ell(\FF_\ell)$. Therefore, $\det(I-B)=0$ if and only if $\det(I-\overline{\rho}_{A,\ell}(\Frob_\p))=0$. Since 
\[
\det(I-\overline{\rho}_{A,\ell}(\Frob_\p))\equiv P_{A,\p}(1)\equiv |A(\FF_\p)| \pmod{\ell},
\]
we find that $\det(I-B)=0$ if and only if $\ell$ divides $|A(\FF_\p)|$.
Similarly, $\det(I-B')=0$ if and only if $\ell$ divides $|A'(\FF_\p)|$. 
The lemma follows from the assumption that $\rad_\Lambda |A(\FF_\p)|$ divides $\rad_\Lambda |A'(\FF_\p)|$ for all $\p\in S$, i.e.,
for all $\ell\in \Lambda$ and $\p\in S$, if $|A(\FF_\p)|$ is divisible by $\ell$, then so is $|A'(\FF_\p)|$.
\end{proof}

Define 
\[
\vV_{\ell}:=\{(B,B')\in H_\ell:\ \det(I-B)=0\}
\]
and 
\[
\vV'_{\ell}:=\{(B,B')\in H_\ell:\ \det(I-B')=0\};
\]
they are closed subvarieties of $H_\ell$ defined over $\FF_\ell$.  
The above definitions were motivated by Lemma \ref{threeptthree}, which says that 
\begin{equation}\label{heart}
\vV_\ell(\FF_\ell)\cap \overline{\rho}_{A\times A',\ell}(\Gal_K)\subseteq 
\vV'_\ell(\FF_\ell)\cap \overline{\rho}_{A\times A',\ell}(\Gal_K).
\end{equation}

We will first prove the following proposition in \S\ref{pfofthreeptone}; it will be a key ingredient in our proof of Proposition \ref{onepttwo} in \S\ref{pfofonepttwo}.

\begin{proposition}\label{threeptone}
We have $\vV_\ell\subseteq \vV'_\ell$ for infinitely many $\ell\in \Lambda$. 
\end{proposition}

The following lemma says the varieties $\vV_\ell\cap \tT_\ell$ and $\vV'_\ell\cap \tT_\ell$, with $\tT_\ell$ a maximal torus of $H_\ell$, carry enough information to prove Proposition \ref{threeptone}. 

\begin{lemma}\label{reduction}
Take any $\ell\in \Lambda$ such that $H_\ell$ is reductive. Let $\tT_\ell$ be a maximal torus of $H_\ell$. If $\vV_\ell\cap \tT_\ell\subseteq \vV'_\ell \cap \tT_\ell$, then $\vV_\ell \subseteq \vV'_\ell$.
\end{lemma}
\begin{proof}
Take any $(B,B')\in \vV_\ell(\overline{\FF}_\ell)$; we have $\det(I-B)=0$.
By the multiplicative Jordan decomposition, $(B,B')\in H_\ell({\overline{\FF}_\ell})$ can be expressed uniquely in the form $(B_s,B'_s)(B_u,B'_u)$ with commuting $(B_s,B'_s)$ and $(B_u,B'_u)\in H_\ell({\overline{\FF}_\ell})$ such that $(B_s,B'_s)$ is semisimple and $(B_u,B'_u)$ is unipotent. In $H_\ell(\overline{\FF}_\ell)$, $(B_s,B'_s)$ is conjugate to some element $(C_s,C'_s)$ of $\tT_\ell({\overline{\FF}_\ell})$. Note that we have
\[
\det(I-C_s)=\det(I-B_s)=\det(I-B)=0
\] 
and so $(C_s,C'_s)\in \vV_\ell\cap \tT_\ell$. By our assumption that $\vV_\ell\cap \tT_\ell\subseteq \vV'_\ell \cap \tT_\ell$, we have $(C_s,C'_s)\in \vV'_\ell\cap \tT_\ell$ and so 
\[
\det(I-B')=\det(I-B'_s)=\det(I-C'_s)=0.
\]
Hence, $(B,B')\in \vV'_\ell(\overline{\FF}_\ell)$.
Since $(B,B')$ is arbitrary, we have $\vV_\ell(\overline{\FF}_\ell)\subseteq\vV'_\ell(\overline{\FF}_\ell)$ and hence $\vV_\ell\subseteq \vV'_\ell$. 
\end{proof}

\subsection{Strategy}\label{strategies}
We will briefly give some ideas behind the proof of Proposition \ref{threeptone}. We will not use this section later. 
  
For $\ell\in \Lambda$, let $\tT_\ell$ be a maximal torus of $H_\ell$. By Lemma \ref{reduction}, it suffices to prove that $\vV_\ell\cap \tT_\ell \subseteq \vV'_\ell\cap \tT_\ell$. 

Take any irreducible component $C$ of $\vV_\ell\cap \tT_\ell$. We want to show that $C\subseteq \vV'_\ell\cap \tT_\ell$, this will imply $\vV_\ell\cap \tT_\ell \subseteq \vV'_\ell\cap \tT_\ell$ since $C$ is arbitrary. 
Suppose on the contrary that $C\not\subseteq \vV'_\ell\cap C$, since $C$ is irreducible, $\dim(\vV'_\ell\cap C)<\dim(C)$. The main idea is to study the set
\[
\Gamma_\ell:=C(\FF_\ell)\cap \overline{\rho}_{A\times A',\ell}(\Gal_K).
\]
and try to bound the cardinality $\gamma_\ell:=|\Gamma_\ell|$
from below and above and to hope for a contradiction for well-chosen primes $\ell\in \Lambda$ and tori $\tT_\ell$. 

\begin{enumerate}

\item   
Theorem \ref{bgHell}(a) says that the index of $\overline{\rho}_{A\times A',\ell}(\Gal_K)$ in $H_\ell(\FF_\ell)$ is bounded independent of $\ell$. So one might expect $\gamma_\ell$ to be roughly of size $|C(\FF_\ell)|$. Then by an application of the Weil conjectures, one would expect that $|C(\FF_\ell)|$ is roughly equal to $\ell^{\dim(C)}$, assuming $C$ is absolutely irreducible. 
Hence, $\gamma_\ell\gg \ell^{\dim(C)}$ and this gives a lower bound of $\gamma_\ell$ with a constant yet to be controlled.

\item 
By equation (\ref{heart}), we have
\[
\Gamma_\ell 
\subseteq (C\cap \vV'_\ell)(\FF_\ell)\cap \overline{\rho}_{A\times A',\ell}(\Gal_K)
\subseteq (C\cap \vV'_\ell)(\FF_\ell).
\]

Then again from the Weil conjectures one would expect that $|(C\cap \vV'_\ell)(\FF_\ell)|$ is $O(\ell^{\dim(C\cap \vV'_\ell)})$. 
Hence, $\gamma_\ell\ll \ell^{\dim(C\cap \vV'_\ell)}\leq \ell^{\dim(C)-1}$ 
and this gives a upper bound of $\gamma_\ell$ with a constant yet to be controlled.
\end{enumerate}

We need to ensure that the implicit constants of $(1)$ and $(2)$ do not depend on $\ell$;
we then have $\ell^{\dim C}\ll \ell^{\dim C-1}$ where the error term is independent of $\ell$. 
This would then give a contradiction for $\ell$ large enough. 
We will restrict our attention to $\ell$ in an infinite subset $\Lambda_0\subseteq \Lambda$ constructed in \S\ref{Lambda0}.

\subsection{The set $\Lambda_0$}\label{Lambda0}
Suppose $\ell$ is a prime for which $H_\ell$ is reductive and split. Choose a split maximal torus $\tT_\ell\subseteq H_\ell$.

By choosing a basis for $(A\times A')[\ell]$, we can identify $H_\ell$ with an algebraic subgroup of $\GL_{2g+2g',\FF_\ell}$ and we may assume that $\tT_\ell$ lies in the diagonal.
We have identified $\tT_\ell$ with a closed subgroup of the diagonal which we identify with $\mathbb{G}_m^{2g+2g'}$; the diagonal of $\GL_{2g+2g'}$. 

For each $1\leq i\leq 2g$, we define $Z_{\ell,i}$ to be the algebraic subgroup $\tT_\ell\cap\{x_i=1\}$ of $\tT_\ell$. Note that 
\[
\vV_\ell\cap \tT_\ell=\bigcup_{i=1}^{2g} Z_{\ell,i}.
\] 

Let $C$ be any irreducible component (defined over $\FF_\ell$) of $Z_{\ell,i}$. 
Theorem \ref{bgHell}(a) says that the index $[H_\ell(\FF_\ell):\overline{\rho}_{A\times A',\ell}(\Gal_K)]$ is bounded by a number $M_{A\times A'}$ which does not depend on $\ell$. 
For each positive integer $m\leq M_{A\times A'}$, define the subvariety 
$C_m:=\{x\in \tT_\ell:\ x^m\in C\}$ of $\tT_\ell$; note that $\dim C_m=\dim C$. 
 
\begin{definition}
Let $\{V_i\}_{i\in I}$ be a collection of affine varieties with $V_i$ defined over a finite field $\FF_{\ell_i}$. We say that $\{V_i\}_{i\in I}$ has \textbf{bounded complexity} if $V_i$ is isomorphic to a closed subvariety of $\mathbb{A}_{\FF_{\ell_i}}^n$ defined by the simultaneous vanishing of $r$ polynomials in $\FF_{\ell_i}[x_1,\cdots,x_n]$ each of degree at most $D$, where the integers $n,r$ and $D$ can be bounded independent of $i\in I$.  
\end{definition}

\begin{lemma}\label{threeptfour}
There is a positive density subset $\Lambda_0\subseteq \Lambda$ such that the following hold:
\begin{itemize}
\item $H_\ell$ is reductive and split for all $\ell\in \Lambda_0$.
\item For each prime $\ell\in \Lambda_0$, irreducible component $C$ of $Z_{\ell,i}\ (1\leq i \leq 2g)$ and positive integer $m\leq M_{A\times A'}$, the irreducible components of $C_m$ are absolutely irreducible. 
\item The set of varieties $\{C\cap \vV'_\ell\}_{\ell,C}$ has bounded complexity 
with $\ell\in \Lambda_0$ and $C$ ranging over the irreducible components of $Z_{\ell,i}\ (1\leq i \leq 2g)$.
\item The set of varieties $\{C_m\}_{\ell,C,m}$ has bounded complexity with $\ell\in \Lambda_0$, $C$ ranging over the irreducible components of $Z_{\ell,i}\ (1\leq i \leq 2g)$ and $m\leq  M_{A\times A'}$.
\end{itemize}
\end{lemma}

\begin{proof}
Fix a number field $F$ and let $\Lambda_0$ be a set consisting of all but finitely many primes $\ell\in \Lambda$ that splits completely in $F$. In our proof, we will allow ourselves to increase $F$ and remove finitely many $\ell$ from $\Lambda_0$. The set $\Lambda_0$ has positive density by the Chebotarev density theorem and our assumption that $\Lambda$ has density $1$. 

By Theorem \ref{bgHell}(c), we can increase $F$ so that $H_\ell$ is reductive and split for all sufficiently large $\ell$ that split completely in $F$. So we may assume that $H_\ell$ is reductive and split for all $\ell\in \Lambda_0$.

Set $M=M_{A\times A'}$. 
Fix $\ell\in \Lambda_0$.
The torus $\tT_\ell$ is the locus in $\mathbb{G}_m^{2g+2g'}$ of a finite set of equations 
\begin{equation}\label{Aelleqn}
\left\{\prod_{i=1}^{2g+2g'}x_i^{n_i}-1:\ (n_1,\cdots,n_{2g+2g'})\in \mathcal{A}_\ell\right\}
\end{equation}
where $\mathcal{A}_\ell$ is a subset of $\ZZ^{2g+2g'}$.
As shown in the proof of \cite[Lemma 3.2]{MR3474973}, we may further assume that $\mathcal{A}_\ell$ is chosen such that $|n_i|\leq B_{A\times A'}$ for all $(n_1,\cdots,n_{2g+2g'})\in \mathcal{A}_\ell$, where $B_{A\times A'}$ is a constant that does not depend on $\ell$.

Let $\mathfrak{T}_{\mathcal{A}_\ell}\subseteq \mathbb{G}_m^{2g+2g'}$ be the subvariety defined over $F$ given by the locus of the set of equations (\ref{Aelleqn}).

For $1\leq i \leq 2g$, let $\mathfrak{Z}_{\ell,i}:=\mathfrak{T}_{\mathcal{A}_\ell}\cap \{x_i=1\}$.
We extend $F$ such that every irreducible component $\mathfrak{C}\subseteq \mathfrak{Z}_{\ell,i}$ is absolutely irreducible. 
For each irreducible component $\mathfrak{C}\subseteq \mathfrak{Z}_{\ell,i}$ and $m\leq M$, we define 
$\mathfrak{C}_m:=\{x\in \mathfrak{T}_{\mathcal{A}_\ell}:\ x^m\in \mathfrak{C}\}$.
We extend $F$ such that every irreducible component of $\mathfrak{C}_m$ is absolutely irreducible. 
We can take our number field $F$ independent of $\ell\in \Lambda_0$ since
there are only finitely many possibilities for $\mathcal{A}_\ell\subseteq \ZZ^{2g+2g'}$.

Suppose $X/F$ is a variety such that all irreducible components are absolutely irreducible. Then
\cite[Lemma (9.7.5)]{MR0217086} says that for any model $\mathcal{X}/\mathcal{O}_F$, the irreducible components of $\mathcal{X}_{\FF_\lambda}$ are also absolutely irreducible for all but finitely many prime ideals $\lambda\subseteq \mathcal{O}_F$. 
Hence, by our choice of $F$ above, for all but finitely many prime ideals $\lambda\subseteq \mathcal{O}_F$, every irreducible component of $(\mathfrak{Z}_{\ell,i})_{\FF_\lambda}\ (1\leq i\leq 2g)$ is absolutely irreducible. Moreover, by further excluding finitely many $\lambda$, for each irreducible component $\mathfrak{C}$ of $\mathfrak{Z}_{\ell,i}$ and $m\leq M$, the irreducible components of $(\mathfrak{C}_m)_{\FF_\lambda}$ are absolutely irreducible.\\

Choose a prime ideal $\lambda|\ell$ of $\mathcal{O}_F$. Since $\ell$ splits completely in $F$, we have $\FF_\lambda = \FF_\ell$. By our choice of $\mathcal{A}_\ell$, the torus $(\mathfrak{T}_{\mathcal{A}_\ell})_{\FF_\lambda}$ is equal to $\tT_\ell$ over $\FF_\lambda = \FF_\ell$. Similarly, for each $1\leq i \leq 2g$, we have an equality $(\mathfrak{Z}_{\ell,i})_{\FF_\lambda} = Z_{\ell,i}$ of varieties over $\FF_\ell$.

Take any irreducible component $\mathfrak{C}$ of $\mathfrak{Z}_{\ell,i}$. After removing a finite number of primes from $\Lambda_0$, we may assume that $C:=(\mathfrak{C})_{\FF_\lambda}$ is an absolutely irreducible variety defined over $\FF_\ell$. In fact, every irreducible component of $Z_{\ell,i}$ arises from such a $\mathfrak{C}$. 
For any $m\leq M$, we have $C_m = (\mathfrak{C}_m)_{\FF_\lambda}$. By removing a finite number of primes from $\Lambda_0$, we may assume that the irreducible components of $C_m$ are absolutely irreducible. 

Note that there are only finitely many $\mathfrak{C}$ and $\mathfrak{C}_m$ as we vary $\ell\in \Lambda_0$ and $m\leq M$ since there are only finitely many $\mathcal{A_\ell}$. So the complexity of all $C$ and $C_m$ is bounded. Moreover, since $C\cap \vV'_\ell= \bigcup_{i=2g+1}^{2g'} C\cap \{x_i=1\}$, the complexity of all $C\cap \vV'_\ell$ is also bounded. 
\end{proof}

\subsection{Proof of Proposition \ref{threeptone}}\label{pfofthreeptone}
Let $\Lambda_0$ be a set of positive density as in Lemma \ref{threeptfour}.
Fix $\ell\in \Lambda_0$. By Lemma \ref{threeptfour}, $H_\ell$ is split. 
Let $\tT_\ell\subseteq H_\ell$ be a split maximal torus and we use the same setup as in \S\ref{Lambda0}.
By Lemma \ref{reduction}, it suffices to prove that $\vV_\ell\cap \tT_\ell \subseteq \vV'_\ell\cap \tT_\ell$.

Suppose that $\vV_\ell\cap \tT_\ell\not\subseteq \vV'_\ell\cap \tT_\ell$; we want to get a contradiction when $\ell\in \Lambda_0$ is large enough. There exists an irreducible component $C$ of $\vV_\ell\cap \tT_\ell$ such that $C\not=\vV'_\ell\cap C$. Let $d$ be the dimension of $C$. Since $C$ is irreducible, the dimension $d'$ of $\vV'_\ell\cap C$ is strictly less than $d$. 
Define 
\[
\Gamma_\ell := C(\FF_\ell)\cap \overline{\rho}_{A\times A',\ell}(\Gal_K)
\]
and $\gamma_\ell:=|\Gamma_\ell|$.

The following lemma is an application of the Weil conjectures, which approximates the cardinality of $\FF_\ell$-points of an affine variety $V$ defined over $\FF_\ell$. 

\begin{lemma}\label{pointcounting}
Let $\{V_i\}_{i\in I}$ be a collection of affine varieties with $V_i$ defined over a finite field $\FF_{\ell_i}$ for each $i\in I$. 
Suppose $\{V_i\}_{i\in I}$ has bounded complexity.
\begin{enumerate}[(a)]
\item For all $i\in I$, we have
\[
|V_i(\FF_{\ell_i})| = O(\ell_i^{\dim V_i})
\] 
where the implicit constant is independent of $i\in I$. 

\item 
Fix an $i\in I$. Suppose that the top dimensional irreducible components of $V_i$ are absolutely irreducible. Then 
\[
|V_i(\FF_{\ell_i})|\geq \ell_i^{\dim V_i} + O(\ell_i^{\dim V_i-1/2})
\]
where the implicit constant is independent of $i$. 
\end{enumerate}
\end{lemma}
\begin{proof}
Let $V\subseteq \mathbb{A}_{\FF_\ell}^n$ with $n>1$ be a closed subvariety defined by the simultaneous vanishing of $r$ polynomials in $\FF_\ell[x_1,\cdots,x_n]$ each of degree at most $D$. Let $b$ be the number of top dimensional irreducible components of $V_{\overline{\FF}_\ell}$.
In \cite[Theorem 2.1]{MR3474973}, Zywina gave the following inequalities:
\begin{equation}\label{DZ1}
|V(\FF_\ell)|\leq b\ell^{\dim V} + 6(3+rD)^{n+1}2^r\ell^{\dim V - 1/2}.
\end{equation}

Suppose further that these components are all defined over $\FF_\ell$. Then 
\begin{equation}\label{DZ2}
\left||V(\FF_\ell)|-b\ell^{\dim V}\right|\leq 6(3+rD)^{n+1}2^r\ell^{\dim V - 1/2}.
\end{equation}

We claim that $b$ is bounded in terms of $n,r$ and $D$ only.
The number $b$ of top dimensional irreducible components of $V_{\overline{\FF}_\ell}$ is equal to the dimension of the $\ell'$-adic \'{e}tale cohomology group $H_c^{2n}(V_{\overline{\FF}_\ell},\QQ_{\ell'})$ with compact support for a prime $\ell'\not= \ell$. Katz \cite[Theorem 1]{MR1803934} showed that $\dim_{\QQ_{\ell'}}H_c^{2n}(V_{\overline{\FF}_\ell},\QQ_{\ell'})$ can be bounded in terms of $n,r$ and $D$ only. The claim is now clear.

Recall that we assumed $\{V_i\}_{i\in I}$ has bounded complexity, i.e., the numbers $n_i,r_i,D_i$ as described above for each $V_i$ are bounded independent of $i$ and hence so is $b_i$ in inequalities \ref{DZ1} and \ref{DZ2} above. Now (a) follows by applying inequality \ref{DZ1} to each $V_i$ and (b) follows by applying inequality \ref{DZ2} to our chosen $V_i$ and using that $b_i\geq 1$. 
\end{proof}

We will now give a lower bound for $\gamma_\ell$.
Set $m_\ell:=[H_\ell(\FF_\ell):\overline{\rho}_{A\times A',\ell}(\Gal_K)]$. By Theorem \ref{bgHell}(a), there exists a constant $M:=M_{A\times A'}$ not depending on $\ell$ such that $m_\ell\leq M$. 
Consider the function 
\[
\varphi: C_{m_\ell}(\FF_\ell)\to \Gamma_\ell,
\ \  g\mapsto g^{m_\ell};
\]
it is well defined since for all $h\in H_\ell(\FF_\ell)$ we have $h^{m_\ell}\in \overline{\rho}_{A\times A',\ell}(\Gal_K)$. 
Since $\tT_\ell$ is a split torus of dimension at most $2g+2g'$, the kernel of $\varphi$ has cardinality bounded by $m_\ell^{2g+2g'}\leq M^{2g+2g'}$.
Since $\ell\in \Lambda_0$, the (top dimensional) irreducible components of each $C_{m_\ell}$ are all absolutely irreducible by Lemma \ref{threeptfour}. 
Hence, by Lemma \ref{pointcounting}(b), we have
\begin{equation}\label{lowerbound}
\gamma_\ell
=|\Gamma_\ell|\geq
\frac{|C_{m_\ell}(\FF_\ell)|}{M^{2g+2g'}}\geq
\frac{\ell^d}{M^{2g+2g'}}+\frac{O(\ell^{d-1/2})}{M^{2g+2g'}}
\end{equation}
where the error term is independent of $\ell$ since the collection of varieties $\{C_{m_\ell}\}_{\ell\in \Lambda_0,C}$ has bounded complexity by Lemma \ref{threeptfour}.  
Inequality (\ref{lowerbound}) gives our lower bound of $\gamma_\ell$.

We will now give a upper bound for $\gamma_\ell$. 
Recall from equation (\ref{heart}) that we have
\[
\vV_\ell(\FF_\ell)\cap \overline{\rho}_{A\times A',\ell}(\Gal_K)\subseteq 
\vV'_\ell(\FF_\ell)\cap \overline{\rho}_{A\times A',\ell}(\Gal_K)
\]
and so 
\[
\Gamma_\ell
=C(\FF_\ell)\cap \overline{\rho}_{A\times A',\ell}(\Gal_K)
\subseteq 
(C\cap \vV'_\ell)(\FF_\ell)\cap \overline{\rho}_{A\times A',\ell}(\Gal_K)
\subseteq (C\cap \vV'_\ell)(\FF_\ell)
\]

Recall that $C\cap\vV'_\ell$ has dimension $d'\leq d-1$. 
Hence, by Lemma \ref{pointcounting}(a), we have 
\begin{equation}\label{upperbound}
\gamma_\ell = O(\ell^{d'})
\end{equation}
where the error term is independent of $\ell$ since the collection of varieties $\{C\cap \vV'_\ell\}_{\ell\in\Lambda_0,C}$ has bounded complexity by Lemma \ref{threeptfour}. Inequality (\ref{upperbound}) gives our upper bound of $\gamma_\ell$.

By combining inequalities (\ref{lowerbound}) and (\ref{upperbound}), we obtain
\begin{equation}\label{final}
\frac{\ell^d}{M^{2g+2g'}}+\frac{O(\ell^{d-1/2})}{M^{2g+2g'}}
= \gamma_\ell
= O(\ell^{d'})
\end{equation}
where the error terms are independent of $\ell$.
In particular, $\ell^d=O(\ell^{d'})$. 
By removing a finite number of primes from $\Lambda_0$, this will contradict $d'<d$.  
Therefore, $\vV_\ell\cap \tT_\ell \subseteq \vV'_\ell\cap \tT_\ell$.
This completes the proof of Proposition \ref{threeptone}.

\subsection{Proof of Proposition \ref{onepttwo}}\label{pfofonepttwo}

Take any prime ideal $\p\in S$. We need to show that $\rad P_{A,\p}(x)$ divides $\rad P_{A',\p}(x)$.

\begin{lemma} \label{dividepoly}
Suppose $f(x)$ and $g(x)\in \ZZ[x]$ are both monic such that the roots in $\overline{\FF}_\ell$ of $f(x)$ are also roots of $g(x)$ for infinitely many $\ell$.
Then $\rad(f)$ divides $rad(g)$.
\end{lemma}
\begin{proof}
Suppose that $\rad(f)$ does not divide $\rad(g)$ and hence there exists an $\alpha\in \overline{\QQ}$ such that $f(\alpha)=0$ and $g(\alpha)\not=0$. Let $F/\QQ$ be a finite Galois extension containing $\alpha$ and all the roots of $g(x)$.
Define 
\[
d:=N_{F/\QQ}\left( \prod_{\beta\in F,\ g(\beta)=0}(\alpha-\beta)\right).
\]
Since $f,g\in \ZZ[x]$ are monic and $g(\alpha)\not=0$, $d$ is a non-zero integer.   
From the assumption of the lemma, there is a prime $\ell\nmid d$ for which the roots in $\overline{\FF}_\ell$ of $f(x)$ are also roots of $g(x)$. Take any prime ideal $\mathcal{L}\subseteq \mathcal{O}_F$ dividing $\ell$. 
For $a\in\mathcal{O}_F$, let $\overline{a}$ be its image in $\mathcal{O}_F/\mathcal{L}$. 
Since every root in $\overline{\FF}_\ell$ of $f(x)$ is also a root of $g(x)$,
we have $\prod_{\beta\in F,\ g(\beta)=0} (\overline{\alpha}-\overline{\beta})=0$ and hence
\[
\prod_{\beta\in F,\ g(\beta)=0} (\alpha-\beta)\in \mathcal{L}.
\]
Therefore, $d\in N_{F/\QQ}(\mathcal{L})\subseteq\ell \ZZ$ which contradicts $\ell\nmid d$. We conclude that $\rad(f)$ divides $\rad(g)$.
\end{proof}

Take $\ell\in\Lambda$ to be any of the infinitely many primes from Proposition \ref{threeptone} such that $\vV_\ell\subseteq \vV'_\ell$ and $\p\nmid \ell$. 
By Theorem \ref{bgHell}(b), we may further assume that $H_\ell$ contains the group $\mathbb{G}_m$ of homotheties.

We claim that the roots in $\overline{\FF}_\ell$ of $P_{A,\p}(x)$ are also roots of $P_{A',\p}(x)$. 
Set $(B,B'):=\overline{\rho}_{A\times A',\ell}(\Frob_\p)\in H_\ell(\FF_\ell)$. 
Suppose that $\lambda\in\overline{\FF}_\ell^\times$ is any root of $\det(xI-B)\equiv P_{A,\p}(x) \pmod{\ell}$.
Since $\mathbb{G}_m\subseteq H_\ell$, we have $(\lambda^{-1}B,\lambda^{-1}B')\in H_\ell(\overline{\FF}_\ell)$. Since $\det(I-\lambda^{-1}B)=0$, we have 
$(\lambda^{-1}B,\lambda^{-1}B')\in \vV_\ell(\overline{\FF}_\ell)$.
By our choice of $\ell$, we have $\vV_\ell\subseteq \vV'_\ell$ and thus 
$(\lambda^{-1}B,\lambda^{-1}B')\in \vV'_\ell(\overline{\FF}_\ell)$.
We deduce that $\lambda$ is also a root of $\det(xI-B')\equiv P_{A',\p}(x)\pmod{\ell}$. This proves our claim. 

Since $P_{A,\p}(x)$ and $P_{A',\p}(x)$ are monic and the roots in $\overline{\FF}_\ell$ of $P_{A,\p}(x)$ are also roots of $P_{A',\p}(x)$ for infinitely many $\ell$, Lemma \ref{dividepoly} implies that $\rad P_{A,\p}(x)$ divides $\rad P_{A',\p}(x)$. This proves Proposition \ref{onepttwo}.

\section{Frobenius Polynomials and Weights}\label{secondpart}
Let $A$ and $A'$ be simple and non-isogenous abelian varieties defined over a number field $K$ of dimensions $g$ and $g'$ respectively. Assume that $K_{A\times A'}^{\text{conn}}=K$, equivalently, the $\ell$-adic monodromy groups $G_{A\times A',\ell}$ are connected. Note that in particular, the $\ell$-adic monodromy groups $G_{A,\ell}$ and $G_{A',\ell}$ are connected.  
We will prove the following theorem in \S\ref{pfoffourptone}.

\begin{theorem}\label{fourptone}
The polynomials $P_{A,\p}(x)$ and $P_{A',\p}(x)$ are relatively prime for almost all $\p\in \Sigma_K$.
\end{theorem}

\begin{remark}
Theorem \ref{fourptone} is false without the connectedness assumption. For example, if $A$ and $A'$ are two non-isogenous CM elliptic curves over $\QQ$, then $P_{A,p}(x) = x^2 + p = P_{A',p}(x)$
for a set of primes $p$ of positive density.
\end{remark}

\subsection{Weights for non-isogenous abelian varieties}\label{weightssection}

Set $G=G_{A\times A',\ell}$; it is connected and reductive.
Fix a maximal torus $T\subseteq G$.  
Let $\Omega_{A,\ell}\subseteq X(T)$ and $\Omega_{A',\ell}\subseteq X(T)$ be the weights of $G$ acting on $V_\ell(A)$ and $V_\ell(A')$ respectively. 
Note that
\[
V_\ell(A\times A')=V_\ell(A)\oplus  V_\ell(A').
\]
Let $W=W(G,T)=N_G(T)(\overline{\QQ}_\ell)/T(\overline{\QQ}_\ell)$ be the absolute Weyl group of $G$ with respect to $T$;  
it acts on $\Omega_{A,\ell}$ and $\Omega_{A',\ell}$.

\begin{lemma}\label{disjointweight}
The sets $\Omega_{A,\ell}$ and $\Omega_{A',\ell}$ are disjoint. 
\end{lemma}

\begin{proof}
Suppose $\Omega_{A,\ell}\cap\Omega_{A',\ell}\not=\emptyset$. Let $\widetilde{\Omega}$ be the $W$-orbit of an element in $\Omega_{A,\ell}\cap\Omega_{A',\ell}$.
Let 
\[
U\subseteq V_\ell(A)\otimes_{\QQ_\ell} \overline{\QQ}_\ell \subseteq V_\ell(A\times A')\otimes_{\QQ_\ell}\overline{\QQ}_\ell
\]
be an irreducible representation of $G_{\overline{\QQ}_\ell}$ for which $\Omega(U)$ contains an element of $\widetilde{\Omega}$. We have $\widetilde{\Omega}\subseteq \Omega(U)$ since $\Omega(U)$ is stable under the action of $W$.
The representation $U$ is minuscule by Proposition \ref{minuscule}, so $\Omega(U)=\widetilde{\Omega}$ and each weight of $U$ has multiplicity $1$. Denote by $\sigma$ the representation of $G_{\overline{\QQ}_\ell}$ on $U$. Similarly, we can construct an irreducible subrepresentation $\sigma'$ of $V_\ell(A')\otimes_{\QQ_\ell} \overline{\QQ}_\ell$ with weights $\widetilde{\Omega}$ that each have multiplicity $1$. Therefore, for every $t\in T$, we have 
\[
\text{tr}\circ \sigma(t)=\sum_{\alpha\in\tilde{\Omega}} \alpha(t)=\text{tr}\circ \sigma'(t)
\]
for all $t\in T$. Since $G$ is reductive, this implies that $\text{tr}\circ \sigma = \text{tr}\circ \sigma'$ and hence
$\sigma$ and $\sigma'$ are isomorphic. So $V_\ell(A)\otimes_{\QQ_\ell}\overline{\QQ}_\ell$ and $V_\ell(A')\otimes_{\QQ_\ell}\overline{\QQ}_\ell$ have an irreducible representation of $G_{\overline{\QQ}_\ell}$ in common.
Therefore,
\[
\Hom_{{{\QQ}_\ell}[\Gal_K]}(V_\ell(A),V_\ell(A'))\otimes_{\QQ_\ell} {\overline{\QQ}_\ell}=
\Hom_{{\overline{\QQ}_\ell}[\Gal_K]}(V_\ell(A)\otimes_{\QQ_\ell} {\overline{\QQ}_\ell},V_\ell(A')\otimes_{\QQ_\ell} {\overline{\QQ}_\ell})\not = 0.
\]
Since $\Hom_{{{\QQ}_\ell}[\Gal_K]}(V_\ell(A),V_\ell(A'))\not=0$, we deduce by Theorem \ref{bgAGell}(b) that $\Hom(A,A')\not=0$. However, this is impossible since $A$ and $A'$ are simple and non-isogenous. Therefore, $\Omega_{A,\ell}$ and $\Omega_{A',\ell}$ are disjoint.
\end{proof}

\subsection{Proof of Theorem \ref{fourptone}}\label{pfoffourptone}
Fix notation as in $\S\ref{weightssection}$.
By Lemma \ref{disjointweight}, we have $\Omega_{A,\ell}\cap \Omega_{A',\ell}=\emptyset$. 
Define 
\[
\mathcal{Z}:=\left\{t\in T:\ \prod_{\underset{\alpha\not=\beta}{\alpha, \beta\in\Omega_{A,\ell}\cup\Omega_{A',\ell}}}\ \left(\alpha(t)-\beta(t)\right)=0 \right\}; 
\]
it is a subvariety of $T$ defined over $\QQ_\ell$ since $\Gal_{\QQ_\ell}$ acts on $\Omega_{A,\ell}\cup\Omega_{A',\ell}$. Moreover, $\dim \mathcal{Z}<\dim T$ since $T$ is irreducible and $\mathcal{Z}\not=T$ ($\Omega_{A,\ell}$ and $\Omega_{A',\ell}$ are non-empty and disjoint, so $\#(\Omega_{A,\ell}\cup \Omega_{A',\ell})\geq 2$).

For each $\p\in\Sigma_K$ for which $A$ and $A'$ have good reduction and $\p\nmid \ell$, choose $t_\p\in T({\overline{\QQ}_\ell})$ such that $t_\p$ is conjugate to $\rho_{A\times A',\ell}(\Frob_\p)$ in $G({\overline{\QQ}_\ell})$.

\begin{lemma}\label{Z}
For almost all $\p\in\Sigma_K$, we have $\alpha(t_\p)\not=\beta(t_\p)$ for all $\alpha,\beta\in \Omega_{A,\ell}\cup\Omega_{A',\ell}$ with $\alpha\not=\beta$.
\end{lemma}

\begin{proof}
Note that $G$ acts on the coordinate algebra $\mathcal{A}=\QQ_\ell[G]$ by composing with conjugation, and $\mathcal{A}^G$ is the set of central functions of $G$. Define $G^{\#}:=\text{Spec}(\mathcal{A}^G)$; it is the variety of semisimple conjugacy classes of $G$. 
Denote the natural projection by $\text{cl}:G\to G^{\#}$; it satisfies the property that 
for $g_1,g_2\in G(\overline{\QQ}_\ell)$, $\text{cl}(g_1)=\text{cl}(g_2)$ if and only if $(g_1)_{s}$ and $(g_2)_{s}$ are conjugate in $G(\overline{\QQ}_\ell)$ (recall that $g_s$ is the semisimple component in the multiplicative Jordan decomposition of $g\in G$). 
Furthermore,
for $t_1,t_2\in T(\overline{\QQ}_\ell)$, $\text{cl}(t_1)=\text{cl}(t_2)$ if and only if $w(t_1)=t_2$ for some $w\in W$.
The map $\text{cl}|_T:T\to G^{\#}$ is dominant and $G^{\#}$ can be identified as a quotient of $T$ (often denoted by $T/\!/W$). The subvariety $\mathcal{Z}_{\overline{\QQ}_\ell}$ of $T_{\overline{\QQ}_\ell}$ is stable under the action of $W$ and thus $\mathfrak{Z}=\text{cl}(\mathcal{Z})$ is a subvariety of $G^{\#}$ which is defined over $\QQ_\ell$. 
Define $\mathfrak{V}:=\{B\in G:\ \text{cl}(B)\in \mathfrak{Z}\}$; it is a subvariety of $G$ with dimension strictly less than $\dim G$ and stable under conjugation by $G$.  

Recall that $\rho_{A\times A',\ell}(\Gal_K)$ is open in $G(\QQ_\ell)$. 
Chebotarev's density theorem \cite[\S 2.2, Corollary 2(b)]{MR1484415} then implies that for almost all $\p\in \Sigma_K$, we have $\rho_{A\times A',\ell}(\Frob_\p)\not\in \mathfrak{V}(\QQ_\ell)$ and hence $t_\p\not\in \mathcal{Z}({\overline{\QQ}_\ell})$. Therefore, for almost all $\p\in\Sigma_K$, we have $\alpha(t_\p)-\beta(t_\p)\not= 0$ for all distinct $\alpha,\beta\in \Omega_{A,\ell}\cup \Omega_{A',\ell}$. 
\end{proof} 

Lemma \ref{disjointweight} says that the sets $\Omega_{A,\ell}$ and $\Omega_{A',\ell}$ are disjoint. 
So by Lemma \ref{Z}, $\{\alpha(t_\p):\alpha\in\Omega_{A,\ell}\}\cap\{\beta(t_\p):\beta\in\Omega_{A',\ell}\}=\emptyset$ for almost all $\p\in\Sigma_K$. So, the set of roots of $P_{A,\p}(x)$ and $P_{A',\p}(x)$ in $\overline{\QQ}_\ell$ are disjoint for almost all $\p\in \Sigma_K$. Therefore, the polynomials $P_{A,\p}(x)$ and $P_{A',\p}(x)$ are relatively prime for almost all $\p\in\Sigma_K$.

\section{Proof of Theorem \ref{mainthm}}\label{proofofthm}
Suppose $A$ is a square-free abelian variety defined over a number field $K$ with $K_A^\conn=K$.
Since $A$ is square-free, it is isogenous to a product $\prod_{i\in I}B_i$, where the $B_i$ are pairwise non-isogenous simple abelian varieties defined over $K$. Let $A'$ be an abelian variety over $K$ for which there exists a density $1$ set $S$ of prime ideals of $\Sigma_K$ 
and a density $1$ set $\Lambda$ of rational primes
such that 
\[
\rad_\Lambda |A(\FF_\p)|=\rad_\Lambda |A'(\FF_\p)|
\]
for all $\p\in S$. In particular, note that we are not yet assuming that $A'$ is square-free.

The following 
proposition which we will prove in \S\ref{pfoffinaleone}, says that $A'$ is isogenous to a product of simple factors of $A$ over an explicit extension of $K$. 

\begin{proposition}\label{finaleone}
The abelian variety
$A'$ isogenous to $\prod_{i\in I} B_i^{e_i}$ over $K_{A\times A'}^\conn$ for some $e_i\geq 1$.
\end{proposition}

The following proposition says that $K_{A\times A'}^\conn$ is in fact $K$; we will give a proof in \S\ref{pfoffinaletwo}.

\begin{proposition}\label{finaletwo}
We have $K_{A\times A'}^\conn=K$.
\end{proposition}

Propositions \ref{finaleone} and \ref{finaletwo} imply that $A'$ is isogenous to $\prod_{i\in I} B_i^{e_i}$ over $K_{A\times A'}^\conn=K$ with $e_i\geq 1$. 
Finally, if we further assume that $A'$ is square-free, we deduce that all the $e_i=1$ and hence $A'$ is isogenous to $A$ over $K$. 
This completes the proof of Theorem \ref{mainthm}.

\subsection{Proof of Proposition \ref{finaleone}}\label{pfoffinaleone}
\begin{lemma}\label{densityone}
To prove Proposition \ref{finaleone}, it suffices to prove it in the case where $K_{A\times A'}^\conn=K$.
\end{lemma} 
\begin{proof}
Set $L=K_{A\times A'}^\conn$. 
Note that $A_L$ is square-free since the $B_i$ are simple over $\overline{K}$ (and hence also over $L$) by Lemma \ref{endA}(b).
The $\ell$-adic monodromy groups of $A_L \times A'_L$ are connected.
We need only show that 
\[
\rad_\Lambda |A(\FF_\mathfrak{P})|=\rad_\Lambda |A'(\FF_\mathfrak{P})|
\]
for a density $1$ set $S'$ of $\mathfrak{P}\in \Sigma_L$, since then Proposition \ref{finaleone} (with the assumption $K_{A\times A'}^\conn=K$) would imply that $A'_L$ is isogenous to $\prod_{i\in I} (B_i)_L^{e_i}$ for some $e_i\geq 1$. 

For a density one set of $\mathfrak{P}\in\Sigma_L$, the inertia degree $f(\mathfrak{P}/\p)$ of $\mathfrak{P}$ over $\p:=\mathfrak{P}\cap \mathcal{O}_K\in \Sigma_K$ is $1$.
Indeed, we have
\[
\sum_{\mathfrak{P}\in \Sigma_L, N(\mathfrak{P})\leq x, f(\mathfrak{P}/\mathfrak{P}\cap \mathcal{O}_K)\geq 2}\  1 
\leq [L:K]\sum_{\p\in \Sigma_K, N(\p)\leq \sqrt{x}}\ 1
\leq 2[L:K][K:\QQ]\pi(\sqrt{x})
\leq [L:\QQ]\sqrt{x}
\]
where $N$ is the norm.
Note that when $f(\mathfrak{P}/\p)=1$, we have $\FF_\mathfrak{P}=\FF_\p$ and hence $|A(\FF_\mathfrak{P})|=|A(\FF_\p)|$. Similarly, $|A'(\FF_\mathfrak{P})|=|A'(\FF_\p)|$.
Hence, by our assumption that $\rad_\Lambda |A(\FF_\p)|=\rad_\Lambda |A'(\FF_\p)|$ for all $\p\in S$, we have
\[
\rad_\Lambda |A(\FF_\mathfrak{P})|=\rad_\Lambda |A(\FF_\p)|=\rad_\Lambda |A'(\FF_\p)|=\rad_\Lambda |A(\FF_\mathfrak{P})|
\]
for almost all $\mathfrak{P}\in \Sigma_L$.
\end{proof}

By Lemma \ref{densityone}, we may assume that $K_{A\times A'}^\conn=K$. 
By assumption, we have $\rad_\Lambda |A(\FF_\p)|=\rad_\Lambda |A'(\FF_\p)|$
for all $\p\in S$. 
By applying Proposition \ref{onepttwo} twice,  
we deduce that 
$\rad P_{A,\p}(x)=\rad P_{A',\p}(x)$ for all $\p\in S$.

The abelian variety $A'$ is isogenous to $\prod_{j\in J} {B'_j}^{e_j}$, 
where the $B'_j$ are pairwise non-isogenous simple abelian varieties defined over $K$ and $e_j\geq 1$. 

Suppose there exists $i\in I$ such that $B_i$ is not isogenous to any $B'_j$. 
Theorem \ref{fourptone} implies there is a prime $\p\in S$ such that $\rad P_{B_i,\p}(x)$ is relatively prime to $\rad P_{B'_j,\p}(x)$ for all $j\in J$. Since 
\[
\rad P_{A',\p}(x)=\rad(\prod_{j\in J} P_{B'_j}(x)^{e_j})=\rad(\prod_{j\in J} P_{B'_j}(x)),
\]
 we deduce that $\rad P_{B_i,\p}(x)$ is relatively prime to $\rad P_{A',\p}(x)$. 
This contradicts that $\rad P_{B_i,\p}(x)$ divides $\rad P_{A,\p}(x)=\rad P_{A',\p}(x)$. Hence, we conclude that for each $i\in I$, there exists $j\in J$ such that $B_i$ is isogenous to $B'_j$; such a $j\in J$ is unique since the $B'_j$ are pairwise non-isogenous.  
By a similar argument, for each $j\in J$, there exists a unique $i\in I$ such that $B_i$ is isogenous to $B'_j$.
So there is a bijection $f:I\to J$ such that $B_i$ is isogenous to $B'_{f(i)}$ for all $i\in I$.
Therefore, $A'$ is isogenous to $\prod_{i\in I} B_i^{e_{f(i)}}$.
The proof of Proposition \ref{finaleone} is now complete. 

\subsection{Proof of Proposition \ref{finaletwo}}\label{pfoffinaletwo}
Set $L:=K_{A\times A'}^\conn$.
Suppose $L\not=K$; we want a contradiction.

By Proposition \ref{finaleone}, there exists an isogeny 
\[
\phi: A'_L \to C_L
\]
defined over $L$, where $C:=\prod_{i\in I} (B_i)^{e_i}$ is an abelian variety over $K$ for some $e_i\geq 1$. 
Since $A$ and $C$ have the same simple factors, up to isogeny, we find that the algebraic groups $G_{A,\ell}$ and $G_{C,\ell}$ are isomorphic. Therefore, $G_{C,\ell}$ is connected by the assumption $K_A^\conn = K$. 

\begin{lemma}\label{qandpi}
There exists a prime ideal $\q\in S$ and an algebraic number $\pi\in \overline{\QQ}$ such that 
$P_{A',\q}(\pi)=0$ and $P_{C,\q}(\pi)\not=0$.
\end{lemma}
\begin{proof}
Fix a prime $\ell$. Set $G=G_{C\times A',\ell}$. 
We can view $G$ as a closed algebraic subgroup of $G_{C,\ell}\times G_{A',\ell}$. 
The isogeny $\phi$ induces an isomorphism $V_\ell(A'_L)\isomto V_\ell(C_L)$ of $\QQ_\ell[\Gal_L]$-modules. Using this isomorphism as an identification, we can assume that  
\[
\rho_{C\times A',\ell}(\sigma)=(\rho_{C,\ell}(\sigma),\rho_{A',\ell}(\sigma))=(\rho_{C,\ell}(\sigma),\rho_{C,\ell}(\sigma))
\]
for all $\sigma\in \Gal_L$.
Since $G_{(C\times A')_L,\ell}$ is the Zariski closure of $\rho_{C\times A',\ell}(\Gal_L)$, we have
$G_{(C\times A')_L,\ell}=\{(B,B):B\in G_{C,\ell}\}$; 
note that $G^\circ=G_{(C\times A')_L,\ell}$ since $G_{C,\ell}$ is connected. Therefore, 
\[
G^\circ = \{(B,B):\ B\in G_{A,\ell}\}.
\] 

By our assumption $L\not=K$, we have $G^\circ \subsetneq G$. So there exists a pair $(g,g')\in G(\QQ_\ell)\backslash G^\circ(\QQ_\ell)$ with $g\not= g'$. Since $(g^{-1},g^{-1})\in G^\circ(\QQ_\ell)$, we have $(I,z)\in G(\QQ_\ell)\backslash G^\circ(\QQ_\ell)$ with $z:=g^{-1}g'\not=I$. 
Since $G(\QQ_\ell)/G^\circ(\QQ_\ell)$ is finite, there exists an integer $m$ such that $(I,z^m)=(I,z)^m\in G^\circ(\QQ_\ell)$ and so $I=z^m$, i.e., $z$ has finite order.
Moreover, since $G^\circ$ is a normal subgroup of $G$, for any pair $(g_0,g_0)\in G^\circ$, $(g_0,z^{-1}g_0z)=(I,z)^{-1}(g_0,g_0)(I,z)\in G^\circ$ and so $zg_0=g_0z$. Hence, $z$ commutes with $G_{C,\ell}$. The coset $(I,z)\cdot G^\circ$ of $G^\circ$ in $G$ is given by
\[
(I,z)\cdot G^\circ = \{(B,z\cdot B):\ B\in G_{C,\ell}\};
\]
it is not $G^\circ$ since $z\not= I$. 

Since $\rho_{C\times A',\ell}(\Gal_K)$ is Zariski dense in $G$ and open in $G(\QQ_\ell)$, 
the set 
\[
\{\sigma\in \Gal_K:\ \rho_{C,\ell}(\sigma)=z\cdot \rho_{A',\ell}(\sigma)\}\supseteq \rho_{C\times A',\ell}^{-1}(((I,z)\cdot G^\circ)(\QQ_\ell))
\]
is open in $\Gal_K$. 
By the Chebotarev density theorem, there exists a prime $\q\in \Sigma_K$ such that $\rho_{C,\ell}(\Frob_\q)=z\cdot \rho_{A',\ell}(\Frob_\q)$. Recall that $z\not=I$ has finite order and commutes with $\rho_{A',\ell}(\Frob_\q)$, so $z$ and $\rho_{C,\ell}(\Frob_\q)$ are simultaneously diagonalizable over $\overline{\QQ}_\ell$ and hence there exist a root $\pi\in \overline{\QQ}$ of $P_{A',\q}(x)$ such that $\zeta\pi\in \overline{\QQ}$ is a root of $P_{C,\q}(x)$ for some root of unity $\zeta\not=1$ in $\overline{\QQ}$. By Larsen and Pink \cite[Corollary 1.4]{MR1441234}, we can further assume that $\q$
is chosen so that the roots of $P_{C,\q}(x)$ in $\overline{\QQ}^\times$ generates a torsion-free group (such $\q$ have density $1$ since the group $G_{C,\ell}$ is connected). 
So, in particular $\pi$ is not a root of $P_{C,\q}(x)$ (if it was, then the subgroup of $\overline{\QQ}^\times$ generated by the roots of $P_{C,\q}(x)$ contains 
$\zeta=(\zeta\pi)\cdot \pi^{-1}$ and hence has torsion). 

Therefore, there exists $\q\in \Sigma_K$ and $\pi\in \overline{\QQ}$ such that $P_{A',\q}(\pi)=0$ and $P_{C,\q}(\pi)\not=0$. 
\end{proof}

We now try to find a prime ideal $\p\in S$ and a prime $\ell\in \Lambda$ such that 
\[
P_{A',\p}(1)\equiv 0 \pmod{\ell}\quad \text{ and }\quad P_{C,\p}(1)\not\equiv 0 \pmod{\ell};
\]
this would then imply that $\rad_\Lambda |C(\FF_\p)| \not= \rad_\Lambda |A'(\FF_\p)|$. Since $\rad_\Lambda |A(\FF_\p)|=\rad_\Lambda |C(\FF_\p)|$, this would contradict our assumption that $\rad_\Lambda |A(\FF_\p)|=\rad_\Lambda|A'(\FF_\p)|$. 

Theorem \ref{bgHell}(b) says that $\mathbb{G}_m\subseteq H_\ell$ when $\ell$ is large enough. Moreover, Theorem \ref{bgHell}(a) says that there exists a number $M_{C\times A'}$ not depending on $\ell$ such that $[H_\ell(\FF_\ell):\overline{\rho}_{C\times A',\ell}(\Gal_K)]\leq M_{C\times A'}$ for all $\ell$. So it follows that there is an integer $m\geq 1$ such that 
\[
(\FF_\ell^\times)^m\cdot I\subseteq \overline{\rho}_{C\times A',\ell}(\Gal_K).
\] 
for all $\ell$.
Let $F$ be number field containing an $m$-th root $\pi^{1/m}$ of $\pi$. Let $\ell\in \Lambda$ be a prime that splits completely in $F$; such a prime exists since we assumed $\Lambda$ has density $1$. Take any $\lambda\in \Sigma_F$ such that $\lambda | \ell$; we have $\FF_\lambda=\FF_\ell$. Define $c$ to be the image of $\pi^{1/m}\in \mathcal{O}_F$ in $\FF_\lambda=\FF_\ell$. Without loss of generality, we assume $\ell\in \Lambda$ is chosen large enough so that $c\not= 0$ and $\q\nmid \ell$; note that the image of $\pi$ in $\FF_\lambda = \FF_\ell$ is $c^m$. 

Define
\[
Y:=(c^m)^{-1}\cdot \overline{\rho}_{C\times A',\ell}(\Frob_\q) = ((c^m)^{-1}\cdot \overline{\rho}_{C,\ell}(\Frob_\q),(c^m)^{-1}\cdot \overline{\rho}_{ A',\ell}(\Frob_\q)).
\]

We have $Y\in \overline{\rho}_{C\times A',\ell}(\Gal_K)$ since $c^m\in \overline{\rho}_{C\times A',\ell}(\Gal_K)$ by our choice of $m$.
Recall that we have $P_{A',\q}(\pi)=0$. So $P_{A',\q}(c^m)\equiv P_{A',\q}(\pi)\equiv 0\pmod{\lambda}$, i.e., $c^m$ is an eigenvalue of $\overline{\rho}_{A',\ell}(\Frob_\q)$. Hence, $(c^m)^{-1}\cdot \overline{\rho}_{ A',\ell}(\Frob_\q)$ has $1$ as an eigenvalue and we have $\det(I-(c^m)^{-1}\cdot \overline{\rho}_{ A',\ell}(\Frob_\q))=0$. 

Suppose that $\det(I-(c^m)^{-1}\cdot \overline{\rho}_{C,\ell}(\Frob_\q))=0$. Then $1$ is an eigenvalue of $(c^m)^{-1}\cdot \overline{\rho}_{C,\ell}(\Frob_\q)$ and $c^m$ would then be an eigenvalue of $\overline{\rho}_{C,\ell}(\Frob_\q)$. So, $P_{C,\q}(\pi)\equiv P_{A,\q}(c^m)\equiv 0 \pmod{\lambda}$, i.e., $\lambda$ divides $P_{C,\q}(\pi)\in \mathcal{O}_F$. Since $P_{C,\q}(\pi)\not=0$, this can only happen for finitely many $\ell\in \Lambda$. So we may assume that $\ell\in \Lambda$ is chosen large enough so that
$\det(I-(c^m)^{-1}\cdot \overline{\rho}_{ C,\ell}(\Frob_\q))\not=0$.

Recall that $Y\in \overline{\rho}_{C\times A',\ell}(\Gal_K)$, so by the Chebotarev density theorem, there exists a prime $\p\in S$ such that $Y=\overline{\rho}_{C\times A',\ell}(\Frob_\p)$. By our arguments above, we have chosen $\ell\in \Lambda$ and $\p\in S$ such that 
\[
P_{A',\p}(1)\equiv 0 \pmod{\ell}\quad \text{ and }\quad P_{C,\p}(1)\not\equiv 0 \pmod{\ell}.
\]
That is, $\ell\in \Lambda$ does not divide $|C(\FF_\p)|$ but divides $|A'(\FF_\p)|$. In particular, $\rad_\Lambda |C(\FF_\p)| \not= \rad_\Lambda |A'(\FF_\p)|$. Since $\rad_\Lambda |A(\FF_\p)|=\rad_\Lambda |C(\FF_\p)|$, this contradicts our assumption that $\rad_\Lambda |A(\FF_\p)|=\rad_\Lambda|A'(\FF_\p)|$.
We deduce that $L=K_{A\times A'}^\conn$ equals $K$. The proof of Proposition \ref{finaletwo} is now complete.

\section{Proof of Corollary \ref{maincor}}\label{pfofmaincor}
The abelian variety $A'$ is isogenous to $\prod_{i=1}^s C_i^{e_i}$ with $C_i$ pairwise non-isogenous simple abelian varieties defined over $K$ and $e_i\geq 1$. By removing a finite number of prime ideals of $S$, we may assume that $A',B_1,\cdots, B_r,C_1,\cdots, C_s$ have good reductions for all $\p\in S$. 
Let $J$ be the set of $j\in \{1,\ldots,s\}$ for which $C_j$ is not isogenous to any $B_i$. We need to show that $J=\emptyset$. 

Define 
\[
A'':=\prod_{i=1}^r B_i \times \prod_{j\in J} C_j
\]
which is square-free by our choice of $J$. 
Since, $A$ is isogenous to the abelian subvariety $\prod_{i=1}^r B_i$ of $A''$, we deduce that $|A(\FF_\p)|$ divides $|A''(\FF_\p)|$ for all $\p\in S$.
On the other hand, 
since $|(\prod_{i=1}^r B_i)(\FF_\p)|=|A(\FF_\p)|$ and $|(\prod_{j\in J}C_j)(\FF_\p)|$ divides $|A'(\FF_\p)|$ for all $\p\in S$, we find that
\[
|A''(\FF_\p)|=|(\prod_{i=1}^r B_i)(\FF_\p)|\cdot|(\prod_{j\in J}C_j)(\FF_\p)| \quad\text{ divides }\quad |A(\FF_\p)|\cdot|A'(\FF_\p)|.
\] 
In particular, $\rad_\Lambda |A''(\FF_\p)|$ divides $\rad_\Lambda (|A(\FF_\p)|\cdot|A'(\FF_\p)|)$.
By our assumption that $\rad_\Lambda|A'(\FF_\p)|$ divides $\rad_\Lambda|A(\FF_\p)|$, we have $\rad_\Lambda (|A(\FF_\p)|\cdot|A'(\FF_\p)|) = \rad_\Lambda |A(\FF_\p)|$ for all $\p\in S$.  
Hence, $\rad_\Lambda|A''(\FF_\p)|$ divides $\rad_\Lambda|A(\FF_\p)|$ and so
\[
\rad_\Lambda |A''(\FF_\p)|
=\rad_\Lambda|A(\FF_\p)|
\] 
holds for all $\p\in S$.
Since both $A''$ and $A$ are squarefree, by Theorem \ref{mainthm}, $A''$ is isogenous to $A$ and hence $J=\emptyset$.

\section{The splitting of reductions of an abelian variety}\label{splitting}
Let $A$ be a simple abelian variety defined over a number $K$ such that $K=K_A^\conn$. Since $A$ is simple, $D:=\End(A)\otimes_\ZZ \QQ$ is a division algebra; note that $A_{\overline{K}}$ is simple and $D=\End(A_{\overline{K}})\otimes_\ZZ \QQ$ by Lemma \ref{endA}. Let $E$ be the center of $D$; it is a number field. In particular, $D$ is a central simple algebra over $E$. Define the integers $e:=[D:E]^{1/2}$ and $r=[E:\QQ]$. 

Choose a prime $\ell$ that splits completely in $E$; it exists by the Chebotarev density theorem. Let $\lambda_i\ (1\leq i \leq r)$ be the prime ideals of $\mathcal{O}_E$ that divides $\ell$.
For each $\lambda_i$, let $E_{\lambda_i}$ be the $\lambda_i$-adic completion of $E$. Then we have $E\otimes_\QQ \QQ_\ell=\bigoplus_{i=1}^r E_{\lambda_i}$. Note that the ring $E\otimes_\QQ \QQ_\ell$ acts on $V_\ell(A)$ and commutes with the $\Gal_K$ action. 
If we let $V_{\lambda_i}(A):=V_\ell(A)\otimes_{E\otimes \QQ_\ell}E_{\lambda_i}$, then we have a decomposition
\[
V_\ell(A) = \bigoplus_{i=1}^r V_{\lambda_i}(A)
\]
of $\QQ_\ell[\Gal_K]$-modules. Each $V_{\lambda_i}(A)$ is also an $E_{\lambda_i}[\Gal_K]$-module which can be expressed as a Galois representation
\[
\rho_{A,\lambda_i}:\ \Gal_K\to \Aut_{E_{\lambda_i}}(V_{\lambda_i}(A))=\Aut_{\QQ_\ell}(V_{\lambda_i}(A))
\]
where the equality uses that $E_{\lambda_i}=\QQ_\ell$ since $\ell$ splits completely in $E$.

Our assumption $K_A^{\text{conn}}=K$ and Theorem \ref{bgAGell}(c) imply that the $\ell$-adic monodromy group $G_{A,\ell}$ is connected and reductive. 
Choose a maximal torus $T\subseteq G_{A,\ell}$ and consider the set $\Omega(V_\ell(A))\subseteq X(T)$ of weights of $G_{A,\ell}$ acting on $V_\ell(A)$.
We will denote by $\Omega(V_{\lambda_i}(A))\ (1\leq i\leq r)$ the set of weights of $G_{A,\ell}$ acting on $V_{\lambda_i}(A)$. We have $\Omega(V_\ell(A))=\cup_{i=1}^r \Omega(V_{\lambda_i}(A))$.

By Theorem \ref{bgAGell}(a), we know that $\rho_{A,\ell}:\Gal_K\to \Aut_{\QQ_\ell}(V_\ell(A))$ is semisimple. In the next lemma, we will see that $\rho_{A,\ell}$ decomposes into absolutely irreducible representations in a very special way.

\begin{lemma}\label{technical}
\noindent
\begin{enumerate}[(a)]
\item For each $\lambda_i$, we have an isomorphism 
\[
V_{\lambda_i}(A)\otimes_{\QQ_\ell} {\overline{\QQ}_\ell}\isom e\cdot W_{\lambda_i}
\]
of $(G_{A,\ell})_{{\overline{\QQ}_\ell}}$ representations (equivalently, $\overline{\QQ}_\ell[\Gal_K]$-modules), where the $W_{\lambda_i}$ are irreducible. 
Moreover,
$W_{\lambda_i}\not\isom W_{\lambda_j}$ for $i\not= j$.

\item The weights of the $(G_{A,\ell})_{{\overline{\QQ}_\ell}}$ representation $W_{\lambda_i}$ form a single orbit under the absolute Weyl group action and each weight has multiplicity one.
\item For $i\not=j$, the representations $W_{\lambda_i}$ and $W_{\lambda_j}$ of $(G_{A,\ell})_{\overline{\QQ}_\ell}$ have no common weights; equivalently, $\Omega(V_{\lambda_i}(A))\cap \Omega(V_{\lambda_j}(A))=\emptyset$.
\end{enumerate}
\end{lemma}

\begin{proof}
\noindent
\begin{enumerate}[(a)]
\item 
 
First, we have natural isomorphisms

\[\begin{array}{rcl}
\End(A)\otimes_\ZZ \QQ_\ell & = & D\otimes_\QQ \QQ_\ell\\
 & = & (D\otimes_E E) \otimes_\QQ \QQ_\ell\\
  & = & D\otimes_E (E\otimes_\QQ \QQ_\ell)\\
   & = & \prod_{i=1}^r D\otimes_E E_{\lambda_i}\\
\end{array}.\]

By tensoring $\End(A)\otimes_\ZZ \QQ_\ell$ with ${\overline{\QQ}}_\ell$ over $\QQ_\ell$, we have
\[
\End(A)\otimes_\ZZ \overline{\QQ}_\ell = \prod_{i=1}^r \left(D\otimes_E E_{\lambda_i}\right)\otimes_{\QQ_\ell} {\overline{\QQ}_\ell}.
\]

Note that $D\otimes_E E_{\lambda_i}$ naturally acts on each $V_{\lambda_i}(A)=V_{\ell}(A)\otimes_{(E\otimes_\QQ \QQ_\ell)} E_{\lambda_i}$ and commutes with the Galois action, so, we have an inclusion
\[
\left(D\otimes_E E_{\lambda_i}\right)\otimes_{\QQ_\ell} {\overline{\QQ}_\ell} \injects  \End_{{\overline{\QQ}_\ell}[\Gal_K]}(V_{\lambda_i}(A)\otimes_{\QQ_\ell} {\overline{\QQ}_\ell}).
\]

Moreover, since $V_\ell(A) = \bigoplus_{i=1}^r V_{\lambda_i}(A)$, we have the inclusion
\[
\prod_{i=1}^r\End_{{\overline{\QQ}_\ell}[\Gal_K]}(V_{\lambda_i}(A)\otimes_{\QQ_\ell} {\overline{\QQ}_\ell}) \injects \End_{{\overline{\QQ}_\ell}[\Gal_K]}(V_\ell(A)\otimes_{\QQ_\ell} {\overline{\QQ}_\ell}).
\]

By combining the above, we thus have the following inclusions: 
\begin{equation}\label{inclusion}
\prod_{i=1}^r \left(D\otimes_E E_{\lambda_i}\right)\otimes_{\QQ_\ell} {\overline{\QQ}_\ell} \injects \prod_{i=1}^r \End_{{\overline{\QQ}_\ell}[\Gal_K]}(V_{\lambda_i}(A)\otimes_{\QQ_\ell} {\overline{\QQ}_\ell}) \injects \End_{{\overline{\QQ}_\ell}[\Gal_K]}(V_\ell(A)\otimes_{\QQ_\ell} {\overline{\QQ}_\ell}).
\end{equation}

By Theorem \ref{bgAGell}(b), we have $\End(A)\otimes_\ZZ {\overline{\QQ}_\ell}\isom \End_{{\overline{\QQ}_\ell}[\Gal_K]}(V_\ell(A)\otimes_{\QQ_\ell} {\overline{\QQ}_\ell})$.
So, the homomorphisms in (\ref{inclusion}) are isomorphisms. The isomorphism 
\[
\prod_{i=1}^r \End_{{\overline{\QQ}_\ell}[\Gal_K]}(V_{\lambda_i}(A)\otimes_{\QQ_\ell} {\overline{\QQ}_\ell}) \isom \End_{{\overline{\QQ}_\ell}[\Gal_K]}(V_\ell(A)\otimes_{\QQ_\ell} {\overline{\QQ}_\ell})
\]
shows that $V_{\lambda_i}(A)\otimes_{\QQ_\ell} {\overline{\QQ}_\ell}$ and $V_{\lambda_j}(A)\otimes_{\QQ_\ell} {\overline{\QQ}_\ell}$
have no isomorphic irreducible representations in common for $i\not= j$. 
On the other hand, 
since $\ell$ splits completely in $E$, we have $E_{\lambda_i}=\QQ_\ell$ and $\left(D\otimes_E E_{\lambda_i}\right)\otimes_{\QQ_\ell} {\overline{\QQ}_\ell}\isom D\otimes_E {\overline{\QQ}_\ell}$ is a central simple algebra over ${\overline{\QQ}_\ell}$ (it is a general fact that if $D$ is a central simple algebra with center $E$, then $D\otimes_E L$ is a central simple algebra over $L$ for any field extension $L$ of $E$).
Hence the algebra $\left(D\otimes_E E_{\lambda_i}\right)\otimes_{\QQ_\ell} {\overline{\QQ}_\ell}$ is isomorphic to $M_e({\overline{\QQ}_\ell})$ since ${\overline{\QQ}_\ell}$ is algebraically closed.
The isomorphism $\End_{{\overline{\QQ}_\ell}[\Gal_K]}(V_{\lambda_i}(A)\otimes_{\QQ_\ell} {\overline{\QQ}_\ell})\isom M_e(\overline{\QQ}_\ell)$ and the semisimplicity of the representation $\rho_{A,\ell}$ implies that $V_{\lambda_i}(A)\otimes_{\QQ_\ell} {\overline{\QQ}_\ell}$ is isotypic and moreover it is isomorphic to a direct summand of $e$ copies of an irreducible representation $W_{\lambda_i}$ of $\overline{\QQ}_\ell[\Gal_K]$. The irreducible representations $W_{\lambda_i}$ and $W_{\lambda_j}$, with $i\not= j$, are not isomorphic since $V_{\lambda_i}(A)\otimes_{\QQ_\ell} {\overline{\QQ}_\ell}$ and $V_{\lambda_j}(A)\otimes_{\QQ_\ell} {\overline{\QQ}_\ell}$ are not isomorphic.

\item
The follows from Proposition \ref{minuscule} and part (a).

\item 
Recall that $\rho_{A,\ell}:\Gal_K \to \Aut_{\QQ_\ell}(V_\ell(A))$ induces a representation $\iota_{A,\ell}:G_{A,\ell}\injects \GL_{V_\ell(A)}$.
For each $i$, the $G_{A,\ell}$-action preserves $V_{\lambda_i}(A)$ and induces a representation
$\iota_{A,\lambda_i}: G_{A,\ell}\to \GL_{V_{\lambda_i}(A)}$.

By (a), $\Omega(V_{\lambda_i}(A))$ is equal to the weights of $W_{\lambda_i}$. So by (b), the Weyl group of $(G_{A,\ell})_{\overline{\QQ}_\ell}$ acts transitively on $\Omega(V_{\lambda_i}(A))$. So for $i\not=j$, $\Omega(V_{\lambda_i}(A))$ and $\Omega(V_{\lambda_j}(A))$ are either equal or disjoint. 

Suppose $\Omega(V_{\lambda_i}(A))=\Omega(V_{\lambda_j}(A))$ with $i\not= j$. We have 
\[
\text{tr}\circ \iota_{A,\lambda_i}(t)=e\cdot\sum_{\alpha\in \Omega(V_{\lambda_i}(A))} \alpha(t) = e\cdot\sum_{\alpha\in \Omega(V_{\lambda_j}(A))} \alpha(t)=\text{tr}\circ \iota_{A,\lambda_j}(t)
\]
for all $t\in T$. Since $G_{A,\ell}$ is reductive, this implies that $\text{tr}\circ \iota_{A,\lambda_i}=\text{tr}\circ \iota_{A,\lambda_j}$ and hence $\iota_{A,\lambda_i}$ and $\iota_{A,\lambda_j}$ are isomorphic. Therefore, the representations $V_{\lambda_i}(A)$ and $V_{\lambda_j}(A)$ of $G_{A,\ell}$ are isomorphic.
This contradicts (a) and hence $\Omega(V_{\lambda_i}(A))$ and $\Omega(V_{\lambda_j}(A))$ are disjoint. \qedhere
\end{enumerate} 
\end{proof}

\begin{lemma}\label{e}
Each weight of the representation $V_{\ell}(A)\otimes_{\QQ_\ell} {\overline{\QQ}_\ell}$ of $(G_{A,\ell})_{\overline{\QQ}_\ell}$ has multiplicity $e$. 
\end{lemma}

\begin{proof}
By Lemma \ref{technical}(c), the sets $\Omega(V_{\lambda_i}(A))$ and $\Omega(V_{\lambda_j}(A))$ are disjoint for $i\not=j$. By Lemma \ref{technical}(b), for each $i$, each weight in the representation $W_{\lambda_i}$ of $(G_{A,\ell})_{\overline{\QQ}_\ell}$ has multiplicity $1$. So by Lemma \ref{technical}(a), each weight of $V_\ell(A)\otimes_{\QQ_\ell} {\overline{\QQ}_\ell}$ has multiplicity $e$.
\end{proof}

\subsection{Proof of Theorem \ref{seppower}}\label{pfofseppower}
By Lemma \ref{e}, each weight of the representation $V_\ell(A)\otimes_{\QQ_\ell} {\overline{\QQ}_\ell}$ of $(G_{A,\ell})_{\overline{\QQ}_\ell}$ has multiplicity $e$. 
So we have 
\[
P_{A,\p}(x)=\prod_{\alpha\in \Omega(V_\ell(A))}(x-\alpha(t_\p))^e
\]
for all $\p$ for which $A$ has good reduction and $\p\nmid \ell$, where $t_\p\in T(\overline{\QQ}_\ell)$ is any element conjugate to $\rho_{A,\ell}(\Frob_\p)$ in $G_{A,\ell}(\overline{\QQ}_\ell)$.

By taking $A'=0$ in Lemma \ref{Z}, if we consider all distinct pairs of $\alpha,\beta\in \Omega(V_\ell(A))$, then for almost all $\p$, $\alpha(t_\p)\not=\beta(t_\p)$ for $\alpha\not=\beta$. Therefore, for almost all $\p$, the Frobenius polynomial $P_{A,\p}(x)$ is the $e$-th power of a separable polynomial.

\section{Number of points on abelian varieties}\label{DZ}
We will now prove the following theorem, as promised in \S\ref{intro}, which says that the function $\p\in S\mapsto |A(\FF_\p)|$ determines $A$ up to isogeny.

\begin{theorem}\label{DZthm}
Let $A$ and $A'$ be abelian varieties defined over a number field $K$.
Let $S$ be any density $1$ set of prime ideals $\p$ of $\mathcal{O}_K$ for which $A$ and $A'$ have good reduction. Suppose
\[
|A(\FF_\p)|=|A'(\FF_\p)|
\]
for all $\p\in S$, 
then $A$ is isogenous to $A'$ (over $K$).
\end{theorem}

When $A$ and $A'$ are elliptic curves, Theorem \ref{DZthm} is an immediate consequence of Faltings' theorem since $P_{A,\p}(x)=x^2 - (N(\p)+1-|A(\FF_\p)|)x + N(\p)=P_{A',\p}(x)$ for all $\p\in S$, where $N(\p)=|\FF_\p|$. In higher dimensions, the theorem does not seem to occur in the literature (and in fact is stated as a conjecture in \cite{MR3429320}). The proof below was supplied by David Zywina.

\begin{proof}[Proof of Theorem \ref{DZthm}]
Fix a prime $\ell$. Let $G:=G_{A\times A',\ell}$; we can identify $G$ with a closed algebraic subgroup of $G_{A,\ell}\times G_{A',\ell}$. 

We claim that $\det(I-B)=\det(I-B')$ holds for every $(B,B')\in G(\QQ_\ell)$. Define
\[
Y:=\{(B,B')\in G:\ \det(I-B)=\det(I-B')\};
\]
it is a subvariety of $G$ stable under conjugation. 
To prove the claim it suffices to 
show that $Y=G$. Take any $\p\in S$ such that $\p\nmid \ell$. By assumption we have $|A(\FF_\p)|=|A'(\FF_\p)|$ and so 
\[
\det(I-\rho_{A,\ell}(\Frob_\p))=|A(\FF_\p)|=|A'(\FF_\p)|=\det(I-\rho_{A',\ell}(\Frob_\p)).
\]
Therefore, $\rho_{A\times A',\ell}(\Frob_\p)\in Y(\QQ_\ell)$ for all $\p\in S$ with $\p\nmid \ell$. By the Chebotarev density theorem, the Zariski closure $G$ of $\rho_{A\times A',\ell}(\Gal_K)$ is contained in $Y$. Therefore, $Y=G$ and the claim is now clear.

Fix any $(B,B')\in G(\QQ_\ell)$ and $\lambda\in \QQ_\ell^\times$.  
It is known that $G$ contains the group $\mathbb{G}_m$ of homotheties \cite{MR574307}. 
So $(\lambda^{-1}B,\lambda^{-1}B')\in G(\QQ_\ell)$ and by our claim above, we have 
\[
\det(I-\lambda^{-1}B)=\det(I-\lambda^{-1}B')
\]
for all $\lambda\in \QQ_\ell^\times$ and so 
\[
\lambda^{g'}\det(\lambda I-B)=\lambda^g\det(\lambda I-B').
\]
Hence, we have $x^{g'}\det(xI-B)=x^g\det(xI-B')\in \QQ_\ell[x]$ since their difference is a polynomial with infinitely many roots in $\QQ_\ell$. Since $\QQ_\ell[x]$ is a UFD and $\det(B)\det(B')\not=0$, it follows that $g=g'$ and $\det(xI-B)=\det(xI-B')$. 

So for all $\p\in S$, we have $\rho_{A\times A',\ell}(\Frob_\p)=(\rho_{A,\ell}(\Frob_\p),\rho_{A',\ell}(\Frob_\p))\in G(\QQ_\ell)$ and hence
\[
P_{A,\p}(x)=\det(xI-\rho_{A,\ell}(\Frob_\p))=\det(xI-\rho_{A',\ell}(\Frob_\p))=P_{A',\p}(x).
\]
By Faltings' theorem, we deduce that $A$ is isogenous to $A'$.  
\end{proof}

\def\cprime{$'$}

\end{document}